\documentclass[10pt]{article}
\usepackage{latexsym}
\usepackage{amsmath,amsthm, amssymb}
\usepackage{amssymb,amsfonts}
\usepackage{color}
\usepackage{graphicx}
\usepackage{epsfig}
\usepackage{mathrsfs}
\usepackage{empheq}
\definecolor{purple}{rgb}{0.65, 0, 1}
\definecolor{orange}{rgb}{1,.5,0}

\def\e{\epsilon}
\def\R{\mathbb{R}}

\def\II{{\rm I\kern-0.5exI}}
\def\III{{\rm I\kern-0.5exI\kern-0.5exI}}
\numberwithin{equation}{section}
\newtheorem{theorem}{Theorem}[section]
\newtheorem{remark}[theorem]{Remark}
\newtheorem{lemma}[theorem]{Lemma}
\newtheorem{definition}[theorem]{Definition}
\newtheorem{proposition}[theorem]{Proposition}
\newtheorem{corollary}[theorem]{Corollary}

\begin{document}
\title{The Patlak-Keller-Segel model and its variations
: \\ properties of solutions via maximum principle
}
\author{ Inwon Kim and Yao Yao\thanks{Department of Mathematics, UCLA.  The authors are partially supported by NSF 0700732 and NSF 0970072.}}
\date{}
\maketitle
\begin{abstract}
In this paper  we investigate qualitative and asymptotic behavior of solutions for a class of diffusion-aggregation equations. Most results except the ones in section 3 and 6 concern radial solutions. The challenge in the analysis consists of the nonlocal aggregation term as well as the degeneracy of the diffusion term which generates compactly supported solutions. The key tools used in the paper are maximum-principle type arguments as well as estimates on mass concentration of solutions.
\end{abstract}

\section{Introduction}
In this paper we study solutions of a nonlocal aggregation equation with degenerate diffusion, given by
\begin{equation} \label{pde}
\rho_t = \Delta \rho^m + \nabla \cdot (\rho\nabla(\rho*V)) \hbox{ in } \R^d\times [0,\infty)
\end{equation}
with
initial data $\rho_0\in L^1(\R^d;(1+|x|^2) dx) \cap L^\infty(\R^d)$. Here $m>1$, $d\geq 3$ and $*$ denotes the convolution operator. In the absence of the aggregation term (when $V=0$), our equation becomes the well-known {\it Porous medium equation} (PME):
\begin{equation}\label{pme}
\rho_t - \Delta(\rho^m)=0.
\end{equation}
Note that, formally, the mass of solutions is preserved over time:
$$
\int_{\mathbb{R}^d} \rho(\cdot,0) dx = \int_{\mathbb{R}^d} \rho(\cdot,t) dx\hbox{ for all } t>0.
$$

 Nonlocal aggregation phenomena have been studied in various biological applications such as population dynamics (\cite{bocm}, \cite {bucm}, \cite{gm}, \cite{tbl}) and Patlak-Keller-Segel (PKS) models of chemotaxis (\cite{ks}, \cite{ll}, \cite{p},\cite{flp}). In the context of biological aggregation, $\rho$ represents the population density which is locally dispersed by the diffusion term, while $V$ is the interaction kernel that models the long-range attraction.  Recently, there has been a growing interest in models with degenerate diffusion to include over-crowding effects (see for example \cite{tbl}, \cite{bocm}). Mathematically, the equation models competition between diffusion and nonlocal aggregation. 
 
 \vspace{5pt}
 
 In this paper we consider the following two types of potentials:
   \begin{itemize}
\item[(A)](PKS-model) $V(x)$ is a \emph{Newtonian potential}:
\begin{equation}\label{kernelA}
V(x) =\mathcal{N}:= -\frac{c_d}{|x|^{d-2}},
\end{equation}
where $c_d := \cfrac{1}{(d-2)\sigma_d}$, with $\sigma_d$: the surface area of the sphere $\mathbb{S}^{d-1}$ in $\mathbb{R}^d$.
\item[(B)] (regularized Newtonian potential) 
\begin{equation}\label{kernelB}
V(x) = (\mathcal{N}*h)(x),
\end{equation}
\end{itemize}
where $*$ denotes convolution and $h(x)$ is a radial function in $L^1(\R^d:(1+|x|^2) dx)\cap L^\infty(\R^d)$ which is continuous and radially decreasing. 

\vspace{5pt}

Note that (A)-(B) covers all attractive potentials $V$ with its Laplacian being nonnegative and radially decreasing. The restrictions on $\Delta V$ turn out to be necessary for obtaining the  preservation of radial monotonicity (see Proposition~\ref{counter})  as well as the mass comparison principle in section 5.

For $m > 1$, the dynamics in (\ref{pde}) is governed by the energy functional
\begin{equation}
\mathcal{F}(\rho) = \int_{ \mathbb{R}^d} \left(\frac{1}{m-1}\rho^m +\frac{1}{2}\rho(\rho*V) \right)dx, \label{energy}
\end{equation}
and when $m=1$, the first term in the integrand is replaced by $\rho \log \rho$.
Indeed \eqref{pde} is the gradient flow for  $\mathcal{F}$ with respect to the Wasserstein metric (see for example \cite{ags} and \cite{cmv}). Depending on $m$, the solution of  \eqref{pde} exhibits different behavior. For $1\leq m < 2-2/d$, the problem is {\it supercritical}: the diffusion is dominant at low concentrations and the aggregation is dominant at high concentration. As a result supercritical and critical problems with singular kernels may exhibit finite time blow-up phenomena (\cite{dp}, \cite{hv}, \cite{s1}, {\cite{blcm}).  On the other hand solutions globally exist with small mass and relatively regular initial data, and here the diffusion dominates at large length scale (see \cite{c} and \cite{s2}). Indeed using the entropy dissipation method (\cite{cjmtu}) it is shown that the solutions with small $L^1$ and $L^{(2-m)d/2}$- norms converge to the self-similar Barenblatt profile (\cite{ls1}-\cite{ls2} and \cite{b2}).

On the other hand, in the {\it subcritical} regime ($m>2-2/d$),  the diffusion is dominant at high concentration. For this reason there is a global solution for all mass sizes (\cite{s1}, \cite{bcl}, \cite{brb}). Since aggregation dominates in low concentration, one can show that there are compactly supported stationary solutions  for any mass size (see Proposition~\ref{stat_sol_regularity_decreasing})). In fact there is no uniqueness result for stationary solutions, even for radial solutions, except the well-known result of Lieb and Yau (\cite{ly}) for the PKS model. Furthermore, even for the PKS model, there are few results addressing the qualitative behavior of general radial solutions: this is perhaps due to the fact that entropy methods face challenges due to the strong aggregation term and the generic presence of the free boundary. This motivates our investigation in this paper.  

 \vspace{10pt}

 The main tools in our analysis are various types of comparison principles. While maximum-principle type arguments are natural to parabolic PDEs,  the classical maximum principle does not hold with \eqref{pde} due to the nonlocal aggregation term, and therefore the standard comparison principle and the corresponding viscosity solutions theory do not apply. Instead we establish order-preserving properties of several associated quantities:  the radial monotonicity (section 4), the mass concentration (section 5), and the rearranged mass concentration for non-radial solutions (section 6).

The following existence and uniqueness results will be used throughout our paper.

\begin{theorem}[Theorem 3 and 7 in \cite{brb}. Also see \cite{bs} and \cite{s1}]
Let $V$ be given by (A) and (B) and $d\geq 3$.
Suppose $\rho_0$ is a nonegative function in $L^1(\R^d; (1+|x|^2)dx)\cap L^\infty(\R^d).$  Then for $m>2-2/d$ there exists a unique, uniformly bounded weak solution $\rho$ of \eqref{pde} in $\R^d\times[0,\infty)$ with initial data $\rho_0$.
\end{theorem}

\textbf{Acknowledgments:} 
We thank Scott Armstrong, Jacob Bedrossian and Thomas Laurent for helpful discussions and communications. 

\subsection{Summary of results}

Let us begin with stating properties of radial, stationary solutions of \eqref{pde}:
\begin{theorem}[Properties of radial stationary solutions]\label{theorem1}
Let $V$ be given by (A) or (B) and let $m> 2-\frac{2}{d}$. Let $\rho_A$ be a non-negative radial stationary solution of \eqref{pde} with $\int \rho_A(x) dx = A>0$. Then
\begin{itemize}
\item[(a)] $\rho_A$ is radially decreasing, compactly supported and smooth in its support (Proposition~\ref{stat_sol_regularity_decreasing});
\item[(b)] $\rho_A$ is uniquely determined for any given $A$ (Theorem 2.2 and Theorem~\ref{uniqueness2}).
\end{itemize}
\end{theorem}
When $V$ is given by (A), the uniqueness of radial stationary solution comes from the well-known results of  Lieb and Yau (\cite{ly}). Their proof is based on the fact that the mass function satisfies an ODE with uniqueness properties; this property fails when $V$ is given by (B).  Instead, we look at the dynamic equation \eqref{pde}, and prove uniqueness out of asymptotic convergence towards a stationary solution. A more direct proof of uniqueness and the uniqueness of general (possibly non-radial) stationary solutions are  interesting open questions. We also mention a recent preprint \cite{bdf}, which studies another type of diffusion-aggregation equation: here authors use eigenvalue methods to prove the uniqueness of one-dimensional stationary solutions. 

Next we show several results concerning the qualitative behavior of general (nonradial) solutions, which will be used in the rest of the paper:
\begin{theorem}[Properties of solutions]
Suppose $m>1$. Let $V$ be given by (A) or (B), and let $\rho(x,t)$ be a weak solution to \eqref{pde}, which is uniformly bounded in $\mathbb{R}^d \times [0,T)$. Then the following holds:
\begin{itemize}
\item[(a)] For any $\delta>0$, $\rho$ is uniformly continuous in $\R^d\times [\delta,T)$; (Theorem \ref{unif_cont})
\item[(b)] [Finite propagation property] $\{\rho>0\}$ expands over time period $\tau$ with maximal rate of $C\tau^{-1/2}$ (Theorem \ref{unif_cont});
\item[(c)] If $\rho(\cdot, 0)$ is radial and radially decreasing, then so is $\rho(\cdot,t)$ for any $t\in[0,T)$ (Theorem \ref{decreasing}).
\end{itemize}
\end{theorem}\label{theorem2}

Both properties (b) (the finite propagation property of the general solutions) and (c) (the preservation of radial monotonicity) are new, to the best of the authors' knowledge, for any type of diffusion-aggregation equation.  For the first-order aggregation equation (\eqref{pde} without the diffusion term),  property (c) has been recently shown in \cite{bgl} for the same class of potentials, via the method of characteristics. 

We now turn to the discussion of asymptotic behavior of solutions.

\begin{theorem} [Asymptotic behavior: subcritical regime]Let $V$ be given by (A) or (B), $m>2-\frac{2}{d}$, and let $\rho(x,t)$ be the solution to \eqref{pde} with radial, compactly supported initial data $\rho_0(x)\in L^1(R^d;(1+|x|^2)dx)\cap L^\infty(\R^d)$ which has mass $A$. Let $\rho_A$ be a radial stationary solution with mass $A$. Then
\begin{itemize}
\item[(a)] The support of $\rho$, $\{\rho(\cdot,t)>0\}$ stays inside of a large ball $\{|x|\leq R\}$ for all $t\geq0$, 
where $R$ depends on $m, d, V$ and the initial data $\rho_0$ (Corollary~\ref{stay_in_compact_set});

\item[(b)]  $\rho$ converges to $\rho_A$ exponentially fast in $p$-Wasserstein distance for all $p>1$ (Corollary~\ref{wasserstein}), and $\|\rho(\cdot, t)-\rho_A\|_{L^\infty(\mathbb{R}^d)} \to 0$ as $t\to \infty$ (Corollary~\ref{cor:convergence}).
\end{itemize}
\label{thm:subcritical}
\end{theorem}

The proof of above theorem is based on the {\it mass comparison}, i.e. maximum principle arguments on the mass concentration of solutions (see Proposition 5.3). The mass comparison property have been previously observed for PKS models (\cite{bkln}; also see a recent preprint of \cite{clw}). However the property has not been fully taken advantage of, perhaps because of the success of entropy method for the KS model. 

Our method  also provides interesting results for asymptotic behavior of radial and non-radial solutions in the supercritical regime, when the solution starts from  sufficiently less concentrated initial data in comparison to a re-scaled stationary profile. (For the definition of ``less concentrated than'', see Definition \ref{def:less_concentrated}) We point out that in our result the mass does not need to be small as required in previous literature (e.g. see \cite{b1}), and provides an explicit description of solutions which are ``sufficiently scattered" so that it does not blow up in finite time.

\begin{theorem} [Asymptotic behavior: supercritical regime] \label{above101}
Let $V(x)$ be given by (A) or (B), and let $1<m< 2-\frac{2}{d}$. Assume $\rho_0$ is radially symmetric, compactly supported and has mass $A$. Then there exists a sufficiently small constant $\delta>0$ depending on $d, m, A$ and $V$, such that if 
$$
\rho_0(\lambda) \prec \delta^d \mu_A(\delta \lambda),
$$  
where ``$\prec$'' is defined in Definition \ref{def:less_concentrated} and $\mu_A(\lambda)$ is given in \eqref{stat_sol_as_tau_to_infty},  then the weak solution $\rho$ with initial data $\rho_0$ exists globally and algebraically converges to the Barenblatt profile in rescaled variables (Corollary~\ref{exp_conv_m<2_corollary}).
\end{theorem}

Note that both Theorem \ref{thm:subcritical} and Theorem~\ref{above101} are limited to radial solutions. The asymptotic behavior of non-radial solutions - either in the subcritical regime or in the supercritical regime in terms of the re-scaled variables - remains largely open, even for the Newtonian potential. Nevertheless it is possible to control the $L^p$-norms of non-radial solutions in terms of radial ones, as we state in the next theorem.

Let us recall that, for any nonnegative measurable function $f$ that vanishes at infinity, the {\it symmetric decreasing rearrangement } $f^*$ is given by 
\begin{equation}\label{def:rearrangement}
f^*(x):=\int_0^\infty \chi_{\{f>t\}^*}(x) dt,
\end{equation}
where $\Omega^*$ denotes the symmetric rearrangement of a measurable set $\Omega$ of finite volume in $\R^d$.

\begin{theorem}[Rearrangement comparison and instant regularization]
Suppose $m>1$. Let $V$ be given by (A) or (B). Let $d\geq 3$ and let $\rho$ be the weak solution to \eqref{pde} with initial data $\rho_0(x)\in L^1(R^d;(1+|x|^2)dx)\cap L^\infty(\R^d)$.
\begin{itemize}
\item[(a)]Let $\bar \rho$ be the solution to the symmetrized problem, i.e. $\bar \rho$ is the weak solution to \eqref{pde} with initial data $ \rho_0^*(x)$.  Assume $\bar \rho$ exists for $t\in[0,T)$. Then $\rho^*(\cdot, t) \prec \bar\rho(\cdot, t)$ and $\|\rho(\cdot, t)\|_p \leq \|\bar \rho(\cdot, t)\|_p$ for $0\leq t < T$ and $1<p\leq \infty$ (Theorem \ref{rearrangement} and Corollary \ref{lp_compare}). 
\item[(b)] Suppose $m>2-\frac{2}{d}$, then for every $0<t<1$ we have 
$$
\|\rho(\cdot,t)\|_{L^\infty(\R^d)}\leq c(m,d, A, V) t^{-\alpha},
$$ 
where $A=\int \rho_0 dx$ and $\alpha := \frac{d}{d(m-1)+2}$. (Proposition \ref{instant}).
\end{itemize} \label{rearrangement_thm_intro}
\end{theorem}

Rearrangement results have been obtained before for \eqref{pme} (Chapter 10 of \cite{v}) and for the two-dimensional Keller-Segel model (\cite{dnr}). We largely follow the arguments in \cite{v}. The new component in the proof is the introduction of approximate equations to deal with both the degenerate diffusion and the nonlocal aggregation term.  The $L^\infty$-regularization result is interesting on its own: similar results have been recently obtained for Keller-Segel model in \cite{pv}, by a De-Giorgi type method.

\section{Properties of the radially symmetric stationary solution}
In this section we consider non-negative radially symmetric stationary solutions of (\ref{pde}), given by
\begin{equation}
\frac{m}{m-1} \rho^{m-1} + \rho*V = C \quad\hbox{ in } \{\rho>0\}, \label{stat_sol}
\end{equation}
where we assume $m>2-\frac{2}{d}$, and the constant $C$ may be different in different positive components of $\rho$.  When $V$ is given by (A) or (B), for any mass $A>0$, the existence of a stationary solution $\rho$ with mass $A$ is proven in \cite{l} and \cite{b2}. 

Let us define the {\it mass function} as follows:
 $$
 M(r):=\int_{B(0,r)}\rho(x)dx.
 $$
Since both $\rho$ and $V$ are radially symmetric, we may slightly abuse the notation and write $\rho*V$ as a function of $r$. When $V=\mathcal{N}$, by the divergence theorem and radial symmetry of $\rho$ and $V$ we have
\begin{equation}
\frac{\partial}{\partial r} (\rho*V)(r) = \frac{M(r)}{\sigma_d r^{d-1}}.  \label{nabla_rho*V}
\end{equation}
where $\sigma_d$ is the surface area of the sphere $\mathbb{S}^{d-1}$ in $\mathbb{R}^d$. Similarly, when $V$ is given by (B), for all radially symmetric function $\rho$, we have that  $\rho*V$ is radially symmetric, and
\begin{equation}
\frac{\partial}{\partial r}  (\rho*V)(r) = \frac{\tilde M(r)}{\sigma_d r^{d-1}} , \label{nabla_rho*V_2}
\end{equation}
where  $\tilde M(r):=\int_{B(0,r)}\rho*\Delta V dx.$
Note that in both cases, we have $\partial_r (\rho*V) \geq 0$.

\begin{proposition} Let $V$ given by (A) or (B) and suppose $m> 2-\frac{2}{d}.$ Then there exists a radially symmetric , nonnegative solution $\rho\in L^1(\mathbb{R}^d)$ of \eqref{stat_sol}. Moreover,
(a) $\rho$ is smooth in its positive set;
(b) $\rho$ is  radially decreasing; and
(c) $\rho$ is compactly supported.
\label{stat_sol_regularity_decreasing}
\end{proposition}

\begin{proof}
1. Existence of the stationary solution $\rho$ follows from \cite{l}: the proof is given in the appendix.

2. To show (a) for $V=\mathcal{N}$, note that the right hand side of \eqref{nabla_rho*V} is continuous since 
$
f(r):=\frac{M(r)}{\sigma_d r^{d-1}}
$
 is continuous for all $r>0$,  and $f(r)\to 0$ as $r\to 0$.
By \eqref{nabla_rho*V}, $\rho*V$ is differentiable in the positive set of $\rho$, which implies that $\rho^{m-1}$ (hence $\rho$) is also differentiable in the positive set of $\rho$.  Therefore $\frac{M(r)}{r^{d-1}}$ is now twice differentiable, hence we can repeat this argument and conclude. When $V$ is given by (B), we can apply the same argument on \eqref{nabla_rho*V_2} and conclude.

3.  By differentiating \eqref{stat_sol} we have
\begin{equation}
\frac{m}{m-1} \frac{\partial}{\partial r} \rho^{m-1} =  -\frac{\partial}{\partial r} (\rho*V) \quad\hbox{ in } \{\rho>0\}. \label{stat_sol_diff}
\end{equation}
Due to  \eqref{nabla_rho*V}-\eqref{nabla_rho*V_2} the right hand side of \eqref{stat_sol_diff} is negative, and thus we conclude (b).

4. It remains to check (c). Note that (b) yields that $\rho$ has simply connected support. Hence  \eqref{stat_sol}  yields
$$
\rho(r) = (C-\rho*V(r))^{\frac{1}{m-1}}.
$$
 When $V=\mathcal{N}$ the proof is similar to that of Theorem 5 in \cite{ly}:  since $\rho*V$ vanishes at infinity, we have

\begin{equation}\label{formular1}
 \rho*V(r) = -\int_r^\infty \frac{M(s)}{s^{d-1}} ds = -\frac{M(r)}{(d-2) r^{d-2}} - \int_r^\infty \frac{c_d}{d-2} \rho(s) s ds, 
\end{equation}
where $c_d$ is the volume of a ball with radius $1$ in $\mathbb{R}^d$.  Note that
\begin{equation}\label{observation1}
\rho*V(r)\leq 0\quad\hbox{ and }\quad -\rho*V(r) \sim \frac{1}{r^{d-2}}\hbox{ as } r\to\infty.
\end{equation}
If $C=0$, \eqref{observation1} implies that 
$$
\rho(r) = (-\rho*V(r))^{\frac{1}{m-1}} \sim r^{-\frac{d-2}{m-1}},
$$ where the exponent is greater than $-d$ when $m>2-\frac{2}{d}$, which contradicts the finite mass property of $\rho$. Therefore $C$ must be negative and thus $\rho(r)$ needs to touch zero for some $r$.

When $V=\mathcal{N}*h$,  we have 
$\rho*V = (\rho*\mathcal{N})*h.$   Since  $h\in L^1(\R^d)$ and is radially decreasing, using \eqref{formular1} we have $\rho*V(x) \sim  \frac{1}{|x|^{d-2}}$ as $|x|\to \infty$ as well, hence by same argument as above, we can conclude.
\end{proof}

Next we state the uniqueness of the radial stationary solution when $V=\mathcal{N}$.

\begin{theorem}[\cite{ly}] \label{prop_uniqueness}
Let $V=\mathcal{N}$, and suppose $m>2-\frac{2}{d}$. Then for all choices of mass $A>0$, the radial stationary solution for \eqref{pde} with mass $A$ is unique.  Moreover, the stationary solution is the global minimizer for the free energy functional \eqref{energy}.
\end{theorem}

This theorem follows from a minor modification from the proof of Theorem 5 in \cite{ly}, which proves uniqueness of the stationary solution of a slightly different problem.  Their proof consists of two steps: firstly for a given mass, they first show the global minimizer of  \eqref{energy} is unique, and secondly they prove every radial stationary solution is a global minimizer for some mass.  Theorem \ref{prop_uniqueness} and the homogeneity of $\mathcal{N}$ yields the following:
\begin{corollary} \label{stat_sol_scaling}
Let $V$ and $m$ be as in Theorem~\ref{prop_uniqueness}, and let $\rho_M$ be the radial solution of \eqref{stat_sol} with mass $M$. Then
\begin{equation}
\rho_M(x) = a \rho_1(a^{-\frac{m-2}{2}}x)\hbox{ with } a:= M^{\tfrac{2}{d(m-2+2/d))}}.
\end{equation}
In particular if $ A<B$ then $\max\rho_B \geq \max\rho_A$ and the following dichotomy of behavior is observed.: (see Figure 1). 
\begin{itemize}
\item[(a)] When $m\geq 2$, $\{\rho_A>0\}\subseteq\{\rho_B>0\}$.
\item[(b)] When $2-\frac{2}{d} < m \leq 2$,  $\{\rho_B>0\}\subseteq\{\rho_A>0\}$. 
\end{itemize}
\end{corollary}

\begin{figure}
\begin{center}
\epsfbox{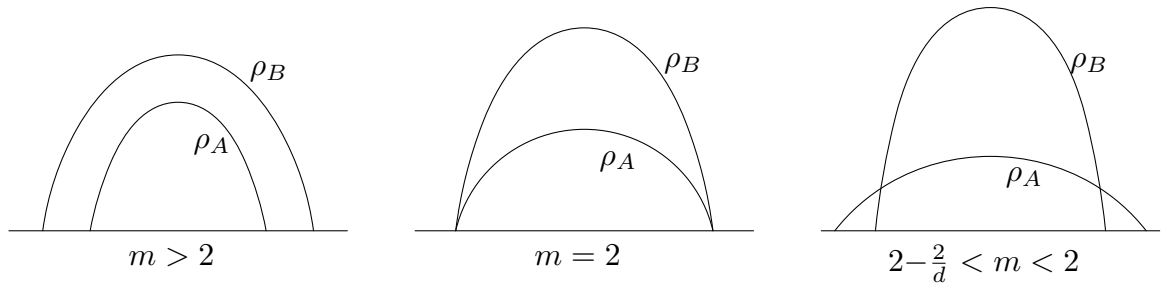}
\end{center}
\vspace{-0.3cm}
\caption{Stationary solutions with different mass for different $m$, where $\int {\rho_A}dx <\int {\rho_B}dx$.}
\end{figure}

We point out that the proof in \cite{ly} cannot be generalized when $V$ is given by (B):  the difficulty  lies in the second step.  When $V=\mathcal{N}$, for any radial stationary solution $\rho$, its mass function $M(r):=\int_{|x|\leq r} \rho(x) dx$ solves a second order ODE
$$
\Big(\frac{m}{m-1} \big(\frac{M'(r)}{\sigma_d r^{d-1}}\big)^{m-1} \Big)'= \frac{M(r)}{\sigma_d r^{d-1}},
$$
where $M(0)=0$ is prescribed. It follows that  $M(r)$ is unique for a given $\rho(0) = \lim_{r\to 0} M'(r)/(\sigma_d r^{d-1})$, which implies that $\rho$ can be uniquely determined by $\rho(0)$. This property then allows both the radial stationary solutions and the global minimizers to be parametrized by their values at the center of mass (see Lemma 12, \cite{ly}).

When $V$ is given by (B), $M(r)$ solves a nonlocal ODE, hence different stationary solutions may have the same center density: thus the above argument in \cite{ly} cannot be applied to prove the second step, necessitating  an alternative approach. Instead of dealing with the stationary equation \eqref{stat_sol} directly, we will consider the dynamic equation \eqref{pde} and prove the uniqueness of the radial stationary solution by their asymptotic convergence. Indeed the following theorem is one of the main results in our paper.

\begin{theorem}\label{uniqueness2}[Corollary~\ref{cor:uniqueness}]
Let $V$ be given by (B), and suppose $m>2-\frac{2}{d}$. Then for any $A>0$, the radial stationary solution of \eqref{pde} with mass $A$ is unique.
\end{theorem}

\section{Qualitative properties of solutions}
 In this section several regularity properties, including the finite propagation property, will be derived for general weak solutions of \eqref{pde}. We point out that the results in this section hold for general (non-radial) solutions.
  
\begin{theorem} \label{unif_cont}
Suppose $m>1$. Let $V$ be given by (A) or (B), and let $\rho$ be a weak solution of (1.1) with its initial data $\rho_0$ bounded with compact support. Further suppose $\rho$ is uniformly bounded in $\mathbb{R}^d \times [0,T]$. Then
\begin{itemize}
\item[(a)] For any $\delta>0$,  $\rho$ is uniformly continuous in $\R^d\times (\delta,T]$.
\item[(b)] [Finite propagation property] Suppose $\{x:\rho(\cdot,t)>0\} \subset B_R(0)$. Then
 $$
 \{x: \rho(\cdot,t+h)>0\}\subset B_{R+Ch^{1/2}}(0) \hbox{ for } 0<h<1,
 $$
 where the constant $C>0$ depends on $m$, $d$, $\rho_0$ and $\|\Delta V\|_1$.
\end{itemize}
\end{theorem}

\begin{proof}
1. Let us consider the case $V=\mathcal{N} $. This is the most singular case and  parallel (and easier) arguments hold for $V$ given by (B).
 Let
$$
C_0 = \sup\{\rho(x,t): (x,t) \in \R^n\times [0,T)\}.
$$
 Observe that treating the convolution term $\Phi:=V*\rho$ as \emph{a priori} given, $\rho$ solves
\begin{equation}\label{fixed_drift}
\rho_t = \Delta (\rho^m) +\nabla\cdot (\rho \nabla \Phi).
\end{equation}
Also,  for all $t\in[0,T)$, $\Phi$ satisfies
\begin{equation}\label{bd1}
| \nabla \Phi|(\cdot,t)  \leq  C_0\int_{|y|\leq 1} | \nabla\mathcal{N}| (y) dy + \|\rho(\cdot,t)\|_1\sup_{|y|\geq 1} |\nabla\mathcal{N}(y)| \leq C_1,
\end{equation}
where $C_1$ depends on $C_0$, the $L^1$ and sup-norm of $\rho$, and the dimension $d$.
Also
\begin{equation}\label{bd2}
|\Delta\Phi|(\cdot ,t) \leq \|\rho\|_{L^{\infty}} \leq C_0 \quad \text{for all }t\in[0,T).
\end{equation}
The bounds \eqref{bd1}-\eqref{bd2} and Theorem 6.1 of \cite{dib}  yields the uniform continuity of $\rho$  in $\mathbb{R}^d\times [\delta,T)$.

2. Next we prove (b). First of all let us point out that the standard comparison principle holds between weak sub- and supersolutions of \eqref{fixed_drift}. For the case of time-independent potential $\Phi(x,t)\equiv \Phi(x)$, Proposition 3.4 of \cite{bh} asserts that the comparison principle between weak sub- and supersolution of \eqref{fixed_drift} holds if the potential function $\Phi$ is independent of the time variable and $|\nabla\Phi|, |\Delta\Phi| \leq C$. The proof in \cite{bh} is based on an approximation of the original problem \eqref{fixed_drift} by a sequence of regularized problems which satisfy the comparison principle (see sections 4, 7, 8 of \cite{bh}). This argument straightforwardly extends to our (time-dependent potential) case, and one can verify that comparison principle between weak sub- and supersolution of \eqref{fixed_drift} holds.

We will now construct a supersolution of \eqref{fixed_drift} to compare with $\rho$ over a small time period to prove the finite propagation property. First observe that the pressure function defined by $u:=\frac{m}{m-1}\rho^{m-1}$  formally satisfies the PDE
\begin{equation}\label{pressure}
u_t = (m-1)u\Delta u + |\nabla u|^2  + \nabla u\cdot \nabla \Phi + (m-1) u\nabla\Phi.
\end{equation} 
Based on this observation, we will first construct a supersolution of \eqref{pressure}, and use the pressure-density transformation to construct the corresponding supersolution of \eqref{fixed_drift}. Let us define
$$
\tilde{U}(x,t):=A\inf_{|x-y|\leq C-Ct}e^{-Ct}(|y|+\omega t-B)_+,
$$
where  $\omega =(1+(m-1)(d-1)) A$, and the constants $B$ and $C$ will be chosen later.

Let $\Sigma:=\{|x|\leq 2B\}\times [0,\omega^{-1}B]$.  Due to Proposition 2.13 in \cite{kl}, $\tilde{U}$ is a viscosity (or weak) supersolution of \eqref{pressure} if $C$ is chosen to be larger than $\max(C_0,C_1)$ given in \eqref{bd1}-\eqref{bd2}. In other words, $\tilde{U}$ satisfies
$$
\tilde{U}_t \geq (m-1)\tilde{U}\Delta\tilde{U} +|\nabla\tilde{U}|^2+ C| \nabla\tilde{U}|+C\tilde{U} \quad\hbox{ in } \{\tilde{U}>0\}\cap\Sigma,
$$ 
 and the outward normal velocity $V_{x,t}$ of the set $\{\tilde{U}>0\}$ at $(x,t)\in\partial\{\tilde{U}>0\}$ satisfies
$$
V_{x,t} = \omega +C \geq  A+C \geq |\nabla\tilde{U}| +C.
$$
Hence $\tilde{\rho}:= (\frac{m-1}{m}\tilde{U})^{1/(m-1)}$ satisfies
$$
\tilde{\rho}_t \geq \Delta(\tilde{\rho}^m)  + C| \nabla\tilde{\rho}| + \frac{C}{m-1}\tilde{\rho} 
$$
 in the domain $\Sigma$, in the viscosity sense (see \cite{kl} for the definition of viscosity solutions of \eqref{fixed_drift}).

Moreover, observe that $\tilde{\rho}^{m-1} \sim \tilde{U}$ is  Lipschitz continuous in space, and continuous in space and time. Using this regularity of $\tilde{\rho}$ as well as the above estimates on the derivatives of $\Phi$, it follows that $\tilde{\rho}$  is a weak supersolution of \eqref{pde} in $\Sigma$, if we choose $C$ greater than $(m-1)C_1$. More precisely the following is true:  for all times $0<t\leq \omega^{-1}B$ and for any smooth, nonnegative function \\$\psi(x,t):\R^n\times (0,\infty) \to \R$ with $\{\psi(\cdot,t)>0\}\subset \{|x|\leq 2B\}$ for $0\leq t\leq \omega^{-1}B$, we have
$$
\int \tilde{\rho}(\cdot,t)\psi(\cdot,t) dx \geq \int \tilde{\rho}(\cdot,0)\psi(\cdot,0) dx +\int \int (\tilde{\rho}^m\Delta\psi + \tilde{\rho}\psi_t -\tilde{\rho}\nabla \Phi\cdot \nabla\psi) dx dt.
$$

Now suppose $\{\rho(\cdot,t_0)>0\}\subset B_R(0) $ for  some $t_0\in[0, T]$. Let us compare  $\rho$  with $\tilde{\rho}$ in \\$\Sigma:=\{|x|\leq 2B\}\times [t_0,t_0+h]$, with
$B= R+h^{1/2}$ and $A = 2C_0h^{-1/2}$. 

Since $\{\rho(\cdot,t_0)>0\}\subset B_R(0)$ and $\rho \leq C_0$, we have $\rho \leq \tilde{\rho}$ at $t=t_0$. Therefore comparison principle asserts that $\rho \leq \tilde{\rho}$ at $t=t_0+h$. In particular 
$$
\{\rho(\cdot, t_0+h)>0\}\subset \{\tilde{\rho}(\cdot,t_0+h)>0\} = B_{R+Mh^{1/2}}(0),
$$
where $M= h^{1/2} + h\omega$ and $\omega= (1+(m-1)(d-1))A$. This proves (b).
\end{proof}
\begin{remark}
Due to  \cite{brb}, when $m>2-\frac{2}{d}$, there exists a global weak solution $\rho$ of (1.1) with initial data $\rho_0$.  Moreover, $\rho$ is uniformly bounded in $\mathbb{R}^d\times (0,\infty)$ due to Theorem 10 in \cite{brb}, so in that case we may let $T=\infty$.
 \end{remark}

 We finish this section with an approximation lemma which links case (A) and (B). Let 
 $$
 h^\e := \e^{-d} h(\frac{x}{\e})
 $$ with $h$ being the standard mollifier in $\R^d$ with unit mass, and let $\rho^\e$ be the corresponding solution of \eqref{pde} with $V=\mathcal{N}*h^\e$ and with initial data $\rho_0$. Then Lemma 8 in \cite{brb} yields that $\{\rho^\e\}_{\e>0
}$ are uniformly bounded for $t\in[0,T]$ for some $T$. This bound as well as Theorem 6.1 of \cite{dib} yields that the family of solutions $\{\rho^\e\}$ are equi-continuous in space and time. This immediately yields the following result:

\begin{proposition}\label{approximation}
Let $\rho_0$ be as given in Theorem~\ref{unif_cont}. Let  $V=\mathcal{N}*h^\e$ and let $\rho^\e$ be the corresponding weak solution of \eqref{pde} with initial data $\rho_0$. Let $\rho$ be the unique solution to \eqref{pde} with $V=\mathcal{N}$ and initial data $\rho_0$, and assume $\rho$ exists for $t\in[0,T)$, where $T>0$ may be infinite. Then the solutions $\rho^\e$ locally uniformly converge to $\rho$ in $\mathbb{R}^d \times [0,T)$.
\end{proposition}

\section{Monotonicity-preserving properties of solutions}
In this section, we show that when $V$ is given by (A) or (B), solutions with radially decreasing initial data remains radially decreasing for all future times.
The main step in the proof is the maximum principle-type argument applied to the double-variable function
$$
\Psi(x,y;t) := \rho(x,t)-\rho(y,t) \hbox{ in } \{|x|\geq |y|\} \times [0,\infty)
$$
to ensure that $\Psi$ cannot achieve a positive maximum at a positive time.

We begin with an observation on the convolution term; the proof is in the appendix.

\begin{lemma} \label{integral_lemma}
Let $V(x)$ be given by (B). Let $u(x)$ be a bounded non-negative radially symmetric function in $\mathbb{R}^d$ with compact support. Further suppose $u(x)$ is not radially decreasing, i.e. there exists $a_1=(\alpha,0,...,0)$ and $b_1=(\beta,0,...,0)$ with $\alpha,\beta>0$ such that
\begin{equation}\label{assumption_ab}
u(b_1)-u(a_1) = \sup_{|a|<|b|} u(b)-u(a) >0.
\end{equation}
Then we have
 $$
 (u*\Delta V)(b_1) - (u*\Delta V)(a_1) \leq \|\Delta V\|_{L^1}(u(b_1)-u(a_1)).
 $$
\end{lemma}

\begin{theorem} \label{decreasing}
Let $V(x)$ be given by (A) or (B).
Suppose that the initial data \\$\rho(x,0)\in L^1(\mathbb{R}^d; (1+|x|^2)dx)\cap L^\infty(\mathbb{R}^d)$ is \emph{radially decreasing}, i.e. $\rho(x,0)$ is radially symmetric and is a decreasing function of $|x|$. We assume a weak solution  $\rho$ exists for $t\in[0,T)$, where $T$ may be infinite. Then $\rho(x,t)$ is radially decreasing for all $t\in[0,T)$.
\end{theorem}
\begin{proof}

1. Without loss of generality we assume that $V$ is given by (B), and $\rho(x,0)$ is positive and smooth. Then a classical solution $\rho(\cdot, t)$ exists for all $t\geq 0$, and we want to show $\rho(\cdot, t)$ is radially decreasing for all $t\geq 0$. When $V=\mathcal{N}$, we can use mollified Newtonian kernel to approximate $\mathcal{N}$; and for general radially decreasing initial data, we can use positive and smooth functions to approximate $\rho(x,0)$. Then the result follows via Proposition \ref{approximation}.  

2. Radial symmetry of $\rho$ for all $t>0$ directly follows from the uniqueness of weak solution.  To prove that $\rho$ is radially decreasing for all time, let us define
\begin{equation*}
w(t) := \sup_{|a|\leq |b|} \rho(b,t)-\rho(a,t).
\end{equation*}
Since $\rho$ is uniformly bounded and uniformly continuous in $\mathbb{R}^d\times [0,\infty)$,  $w(t)$ is continuous in $t$, and uniformly bounded for $t\in[0,\infty)$. Moreover, $\rho(x,0)$ being radially decreasing is equivalent with  $w(0)=0$. We will use a maximum principle-type argument to show that $w(t) = 0$ for all $t\geq 0$, which proves the theorem.

Suppose $w \not \equiv 0$. Then for any $\lambda>0$ the function $w(t)e^{-\lambda t}$ has a positive maximum at $t=t_1$ for some $t_1>0$. We will show that this cannot happen when we choose $\lambda>2\|\rho\|_\infty\|\Delta V\|_1$.

At $t=t_1$ there exists $a_1=(\alpha,0,...,0)$ and $b_1=(\beta,0,...,0)$ such that
 $\alpha <\beta$ and
\begin{equation}\label{achieve_sup}
\rho(b_1,t_1)-\rho(a_1,t_1) = w(t_1)>0. \quad\text{(See Figure \ref{rho_t1}.)}
\end{equation}

\begin{figure}[h]
\begin{center}
\epsfbox{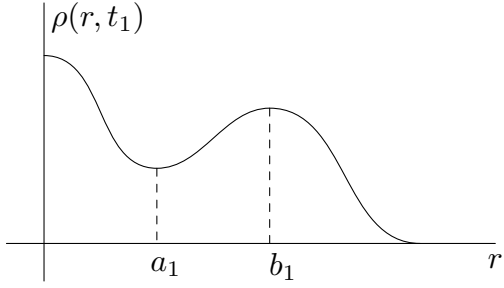}
\end{center} 
\vspace{-0.4cm}
\caption{Graph of $\rho$ at time $t_1$}
\label{rho_t1}
\end{figure}

Moreover by definition $\rho(b_1,t)-\rho(a_1,t) \leq w(t)$, and thus
\begin{equation*}
\frac{d}{dt} ((\rho(b_1,t)-\rho(a_1,t)) e^{-\lambda t}) = 0~~~\text{at }t=t_1,
\end{equation*}
which means
\begin{equation}\label{rho_t_ineq}
\rho_t(b_1,t_1)-\rho_t(a_1,t_1) = \lambda (\rho(b_1,t_1)-\rho(a_1,t_1)).
\end{equation}
 Further observe that $\rho(\cdot,t_1)$ has a local minimum (in space only) at $a_1$ and a local maximum at $b_1$. This yields
\begin{equation*}
\nabla \rho(a_1,t_1) = 0\quad{ and } \nabla \rho(b_1, t_1)=0,
\end{equation*}
as well as 
\begin{equation*}
\Delta \rho^m(a_1, t_1)\geq 0 \hbox{ and } \Delta \rho^m(b_1, t_1)\leq 0.
\end{equation*}

Let us now make use of the equation (1.1) that $\rho$ satisfies to get a contradiction: we have
\begin{eqnarray}\label{hoho}
\nonumber\rho_t(b_1,t_1)-\rho_t(a_1,t_1) &=& \Delta \rho^m(b_1,t_1)  + \nabla \cdot (\rho\nabla(\rho*V))(b_1,t_1) \\
\nonumber&& - \Delta \rho^m(a_1,t_1) - \nabla \cdot (\rho\nabla(\rho*V))(a_1,t_1)\\[0.3cm]
\nonumber&\leq& \rho(b_1,t_1) (\rho*\Delta V)(b_1,t_1) - \rho(a_1,t_1) (\rho*\Delta V)(a_1,t_1)\\[0.3cm]
&=&\rho(b_1, t_1)[(\rho*\Delta V)(b_1,t_1)-(\rho*\Delta V)(a_1,t_1)]\label{rho_t_ineq2}\\
\nonumber&& + (\rho(b_1,t_1)-\rho(a_1,t_1)) (\rho*\Delta V)(a_1,t_1). 
\end{eqnarray}

In order to bound the first term in \eqref{rho_t_ineq2}, we apply Lemma \ref{integral_lemma},which gives
\begin{equation}\label{xx1}
(\rho*\Delta V)(b_1,t_1)-(\rho*\Delta V)(a_1,t_1) \leq \|\Delta V\|_1(\rho(b_1,t_1)-\rho(a_1,t_1)),
\end{equation}
and for the second term we use 
\begin{equation}(\rho*\Delta V)(a_1,t_1) \leq \|\rho\|_\infty \|\Delta V\|_1.\end{equation} Due to the estimates \eqref{xx1}-\eqref{xx2},  \eqref{hoho} yields that
\begin{eqnarray*}
\nonumber\rho_t(b_1,t_1)-\rho_t(a_1,t_1) &\leq&2\|\rho\|_\infty\|\Delta V\|_1 (\rho(b_1,t_1)-\rho(a_1,t_1)),
\end{eqnarray*}
which contradicts \eqref{rho_t_ineq} since we chose $\lambda$ to be strictly greater than $2\|\rho\|_\infty\|\Delta V\|_1$.
\end{proof}

The following proposition states that in the previous theorem, the condition that $\Delta V$ is radially decreasing is indeed necessary.

\begin{proposition}\label{counter}
Let $V(x)$ be radially symmetric, $\Delta V \geq 0$, with $\Delta V$  continuous, but not radially decreasing.   Then there exists a radially decreasing initial
data $\rho_0$  such that the solution $\rho(x,t)$ of (1.1) starting with initial data $\rho_0$ does not preserve the radial monotonicity over time.
\end{proposition}
\begin{proof}
Since $\Delta V$ is not radially decreasing, we can find $x_1, x_2 \in \mathbb{R}^d$, such that $0<|x_1| < |x_2|$, and $\Delta V(x_1) < \Delta V(x_2)$.

For a small $\e>0$, let $\rho_0(x)$ be given as below:
\begin{equation*}
\rho_0(x) = \epsilon \chi_{B(0,x_2+1)}*\phi(x) + \frac{1}{\epsilon^d} \phi(\frac{x}{\epsilon}),
\end{equation*}

where $\chi_E$ is the characteristic function of $E$ and $\phi$ is a radially symmetric mollifier with unit mass and supported in $B(0,r_0)$, where $r_0 < \min\{1,|x_1|/2\}$. Note that in a small space-time neighborhood of $x_1$ and $x_2$, $\rho$ solves a uniformly parabolic equation, and thus is smooth.

Since $\Delta \rho^m(x_i,0)=\nabla \rho(x_i,0) = 0$ for $i=1,2$, we have
\begin{equation*}
\rho_t(x_i,0) = \rho_0(x_i) (\rho_0*\Delta V)(x_i),~~~i = 1,2.
\end{equation*}
Since $\rho_0(x_1) = \rho_0(x_2)$, to show $\rho_t(x_1,0) < \rho_t(x_2,0)$, it suffices to prove
\begin{equation}
(\rho_0*\Delta V)(x_1) < (\rho_0 * \Delta V)(x_2). \label{wts_v}
\end{equation}
Note that $\rho_0 *\Delta V$ locally uniformly converges to $\Delta V(x)$ as $\e\to 0$. Since $\Delta V(x_1) < \Delta V(x_2)$, if we let $ \epsilon$ be sufficiently small, we would have \eqref{wts_v}. In particular $\rho(x_1,t) < \rho(x_2,t)$ for small $t>0$, which means $\rho(x,t)$ stops being radially monotone as soon as $t>0$.
\end{proof}

\section{Mass Comparison and asymptotic behavior for radial solutions} \label{sec_comparison}
Recall that there is no classical comparison principle for (1.1), due to the nonlocal term; however, we will prove that a comparison principle actually hold for the {\it mass function}
\begin{equation}\label{mass}
M(r,t)=M(r,t;\rho):= \int_{B(0,r)} \rho(x,t) dt
\end{equation}
(see Proposition~\ref{comp_concentration}).  We point out that the corresponding property has been observed for \eqref{pme} (\cite{v}) and also for  the Keller-Segel  model (\cite{bkln}). It turns out that mass comparison holds for \eqref{pde} if  the potential $V$ satisfies $\Delta V \geq 0$ in the distribution sense (Proposition~\ref{comp_concentration}.) As we will see later,  mass comparison effectively describes the asymptotic behavior of radial solutions in both sub- and supercritical regime.

\subsection{Mass comparison}
  First note that, if $\rho$ is a weak solution of \eqref{pde}, in the positive set, $\rho$ is a bounded solution of a locally uniformly parabolic, divergence-type equation with continuous coefficients. Hence due to \cite{lsu}, $\rho$ is at least $C^1$ in space and time variables. It follows that $M(r,t)$ is $C^2$ in space and $C^1$ in time in $\{\rho>0\}$.

Let us denote the support of $\rho$ at time $t$ by $B(0,R(t))$. Then we compute the PDE that $M(r,t)$ satisfies in $\{r<R(t)\}$. Due to \eqref{nabla_rho*V_2}, we have
\begin{eqnarray}
\nonumber\frac{\partial M}{\partial t}(r,t) &=& \int_{\partial B(0,r)} \vec{n}\cdot (\nabla \rho^m + \rho \nabla(\rho*V)) dx\\
\nonumber&=& \sigma_d r^{d-1} \Big(\frac{\partial}{\partial r}\big((\frac{\partial M}{\partial r} \frac{1}{\sigma_d r^{d-1}})^m\big) + (\frac{\partial M}{\partial r} \frac{1}{\sigma_d r^{d-1}}) (\frac{\tilde M}{\sigma_d r^{d-1}})\Big)\\
&=& \sigma_d r^{d-1}\frac{\partial}{\partial r} \big((\frac{\partial M}{\partial r} \frac{1}{\sigma_d r^{d-1}})^m\big)  + \big(\frac{\partial M}{\partial r} \frac{1}{\sigma_d r^{d-1}}\big)\tilde M,  \label{pde_for_m}
\end{eqnarray}
where
\begin{equation}\label{tildem}
\tilde M(r,t)=\tilde M(r,t;\rho) := \int_{B(0,r)} (\rho(\cdot, t)*\Delta V)(x) dx.
\end{equation}

\begin{definition}
Let $\rho_1$ and $\rho_2$ be two non-negative radially symmetric functions in $L^1(\mathbb{R}^d)$.  We say $\rho_1$ is \emph{less concentrated than} $\rho_2$, or $\rho_1 \prec \rho_2$, if
$$\int_{B(0,r)} \rho_1(x) dx \leq \int_{B(0,r)} \rho_2(x)dx\quad\hbox{ for all } r\geq 0.$$ \label{def:less_concentrated}
\end{definition}

\begin{definition}
Let $\rho_1(x,t)$ be a non-negative, radially symmetric function in $L^\infty(([0,T];L^1(\mathbb{R}^d)\cap L^\infty(\mathbb{R}^d))$, which is $C^1$ in its positive set. We say $\rho_1$ is a \emph{supersolution} of  \eqref{pde_for_m} if $M_1(r,t):= M(r,t;\rho_1)$ and $\tilde{M}_1:= \tilde{M}(r,t;\rho_1)$ satisfy
\begin{equation} \label{ineq_m1}
\frac{\partial M_1}{\partial t} \geq \sigma_d r^{d-1}\frac{\partial}{\partial r} \Big(\big(\frac{\partial M_1}{\partial r} \frac{1}{\sigma_d r^{d-1}}\big)^m\Big)  + \big(\frac{\partial M_1}{\partial r} \frac{1}{\sigma_d r^{d-1}}\big) \tilde M_1
\end{equation}
in the positive set of $\rho_1$.

Similarly we define a {\it subsolution} of \eqref{pde_for_m}.
\end{definition}

\begin{proposition}[Mass comparison] Suppose $m>1$. Let $V$ be given by (A) or (B), and let $\rho_1(x,t)$  be a supersolution and $\rho_2(x,t)$ be a subsolution of \eqref{pde_for_m} for $t\in [0,T]$. Further assume that the $\rho_i$'s  preserve their mass over time, i.e.,
$\int \rho_1(\cdot, t) dx$ and $\int \rho_2(\cdot, t) dx$ stay constant for all $0\leq t \leq T$.  Then the mass functions are ordered for all time: i.e.,
if  $\rho_2(x,0) \prec \rho_1(x,0)$, then we have 
$$\rho_2(x,t) \prec \rho_1(x,t)\hbox{ for all }t\in [0,T].$$
\label{comp_concentration}
\end{proposition}

\begin{proof}We claim that $M_2(r,t) \leq M_1(r,t)$ for all $r\geq0, t\in [0,T]$, which proves the proposition. At $t=0$, we have $M_2(r,0)\leq M_1(r,0)$ for all $r\geq 0$.  For the boundary conditions of $M_i$, note that
\begin{equation*}
\begin{cases}
M_1(0,t)=M_2(0,t)=0 \quad \text{for all $t\in [0,T]$};\\
\lim_{r\to\infty}(M_1(r,t)-M_2(r,t)) = \int_{\mathbb{R}^d} (\rho_1(x,0) - \rho_2(x,0)) dx\geq 0  \quad \text{for all $t\in [0,T]$}.
\end{cases}
\end{equation*}

To prove the claim, for given $\lambda>0$, we define
$$w(r,t):=\big(M_2(r,t)-M_1(r,t)\big)e^{-\lambda t},$$
where $\lambda$ is a large constant to be determined later. If the claim is false, then $w$ attains a positive maximum at some point $(r_1, t_1) \in (0,\infty)\times(0,T
]$.
Moreover, since the total masses of $\rho_i$'s are preserved over time and thus are ordered, we know that $(r_1,t_1)$ must lie inside the positive set for both $\rho_1$ and $\rho_2$, where the $M_i$'s are $C^{2,1}_{x,t}$.

Since it is assumed that $w$ attains a maximum at $(r_1, t_1)$, the following inequalities hold at this point:
\begin{eqnarray}
&&w_t \geq 0~\Longrightarrow \cfrac{\partial (M_2-M_1)}{\partial t} \geq \lambda (M_2-M_1);\label{compare_t}\\
&&w_r=0~\Longrightarrow \cfrac{\partial M_1}{\partial r} = \cfrac{\partial M_2}{\partial r};\label{compare_r}\\
&&w_{rr}\leq 0 ~\Longrightarrow \cfrac{\partial^2 M_1}{\partial r^2} \geq \cfrac{\partial^2 M_2}{\partial r^2}.\label{compare_rr}
\end{eqnarray}
Now consider the first term on the right hand side of \eqref{ineq_m1}, and the corresponding inequality for $M_2$. \eqref{compare_r} and \eqref{compare_rr} imply that
\begin{equation}\label{ineq:above}
\frac{\partial}{\partial r} \Big((\frac{\partial M_1}{\partial r} \frac{1}{\sigma_d r^{d-1}})^m\Big) \geq \frac{\partial}{\partial r} \Big((\frac{\partial M_2}{\partial r} \frac{1}{\sigma_d r^{d-1}})^m\Big) ~~~~\text{at } (r_1, t_1).
\end{equation}

Subtracting the inequality \eqref{ineq_m1} with the corresponding inequality for $M_2$, and using \eqref{ineq:above}, we obtain
\begin{equation}
\cfrac{\partial (M_2-M_1)}{\partial t} \leq (\frac{\partial M_1}{\partial r} \frac{1}{\sigma_d r^{d-1}}) (\tilde M_2 - \tilde M_1)\label{m2m1}~~~~\text{at } (r_1, t_1).
 \end{equation}

 We next claim $$(\tilde M_2 - \tilde M_1)(r_1, t_1) \leq C(M_2-M_1)(r_1, t_1),$$
 where $C$ only depends on $V$.  For the Newtonian potential this is obvious, since $\tilde M \equiv M$.  For a mollified Newtonian potential, we estimate $(\tilde M_2 - \tilde M_1)(r_1, t_1)$:
  \begin{eqnarray*}
\tilde M_2(r_1, t_1) - \tilde M_1(r_1, t_1) &=& \int_{\mathbb{R}^d} ((\rho_2-\rho_1)*\Delta V) ~\chi_{B(0,r_1)} dx\\
&=& \int_{\mathbb{R}^d} (\rho_2-\rho_1) (\chi_{B(0,r_1)} * \Delta V) dx.
 \end{eqnarray*}
By our assumption, $\Delta V$ is radially decreasing, thus $\chi_{B(0,r_1)} * \Delta V$ is radially decreasing and has maximum less than or equal to $\|\Delta V\|_{1}$.  Therefore we can use a sum of bump functions to approximate the function $\chi_{B(0,r_1)} * \Delta V$, where the sum of the height is less than $\|\Delta V\|_1$.
Hence, like in the proof of Lemma \ref{integral_lemma}, we get
 \begin{equation}\label{estimate1001}
 \tilde M_2(r_1, t_1) - \tilde M_1(r_1, t_1)  \leq \|\Delta V\|_{1} \sup_x (M_2 - M_1)(x, t_1) = \|\Delta V\|_1 (M_2-M_1)(r_1,t_1).
 \end{equation}
 We  plug the estimate \eqref{estimate1001} into \eqref{m2m1}, which then becomes
 \begin{equation}\label{m2m1_2}
 \cfrac{\partial (M_2-M_1)}{\partial t} \leq \rho_1 \|\Delta V\|_1 (M_2-M_1)
~~~~\text{at } (r_1, t_1).
 \end{equation}
 Now choose $\lambda > \sup_{\R^n\times (0,T]} \|\Delta V\|_1 \|\rho_1\|_{L^\infty(\mathbb{R}^d\times [0,T])}$. Then we see that \eqref{m2m1_2} contradicts \eqref{compare_t}.
\end{proof}

Observe that \eqref{pde} can be written as a transport equation
$$
\rho_t + \nabla \cdot (\rho \vec{v})=0,$$
where the \emph{velocity field} $\vec{v}$ is defined by
\begin{equation}\label{vector}
\vec{v}(x,t;\rho) := -\frac{m}{m-1}\nabla(\rho^{m-1})-\nabla(\rho*V).
\end{equation}

Then the mass function for a radial solution of \eqref{pde} satisfies
\begin{equation}\label{observation}
\frac{\partial}{\partial t}M(r,t) = -\rho(r,t) \int_{\partial B(0,r)} (\vec{v}\cdot \vec{n}) ds.
\end{equation}
The above observation along with Proposition~\ref{comp_concentration} immediately yields the following corollary:

\begin{corollary} \label{velocity_field_corollary}
Suppose $m>1$. Let $V$ be given by (A) or (B).  Let $\rho_0(x)$ be a continuous radially symmetric function, which is differentiable in its positive set.  We assume that the velocity field of $\rho_0$ is pointing inside everywhere, i.e., for $\vec{v}$ as defined in \eqref{vector},
\begin{equation}\label{velocity_field_condition}
 \vec{v}(x; \rho_0) \cdot \Big(-\dfrac{x}{|x|}\Big) \geq 0 \quad \text{in } \{\rho_0>0\}.
\end{equation}
Let $\rho$ be the weak solution of \eqref{pde} with initial data $\rho_0 \prec\rho(\cdot,0) $.  Then $ \rho_0\prec \rho(\cdot, t)$ for all $t\geq 0.$
\end{corollary}
\begin{proof}
Let us define
$$\rho_1(x,t) := \rho_0(x) \hbox{ for } (x,t)\in\mathbb{R}^d\times [0,\infty).$$
Then \eqref{observation} and \eqref{velocity_field_condition} yield that $\rho_1$ is a subsolution of \eqref{pde_for_m}.
 Therefore, Proposition~\ref{comp_concentration} applies to $\rho$ and $\rho_1$ and so we are done.
\end{proof}

As an application of Corollary~\ref{velocity_field_corollary}, we will show that when the initial data is radially symmetric and compactly supported, the support of the solution will stay in a large ball for all times.

\begin{corollary}[Compact solutions stay compact]\label{stay_in_compact_set}
Let $V$ be given by (A) or (B), and let $m>2-\frac{2}{d}$.  Let $\rho$ solve \eqref{pde} with a continuous, radially symmetric and compactly supported initial data $\rho(x,0)$. Then there exists $R>0$ depending on $m, d, \|\Delta V\|_1$ and $\rho(\cdot, 0)$, such that $$\{\rho(\cdot,t)>0\}\subset\{|x|\leq R\}
\hbox{ for all }t>0.$$
\end{corollary}

\begin{proof}
1.  We will first assume that $\rho(0,0)>0$.  Let $A:= \int_{\mathbb{R}^d} \rho(x,0) dx$, and let $\rho_A(x)$ be a radial stationary solution with mass $A$.  For any continuous radial initial data with $\rho(0,0)>0$, we can choose $a>0$ sufficiently small, such that
$$
\rho_1(x,t) := a^d \rho_A(ax) \prec \rho(x,0).
$$ 
Our aim is to show that the velocity field of $\rho_1(x,t)$ is pointing towards the inside all the time, i.e.,
\begin{equation}\label{inward_velocity}
v(r,t;\rho_1):=\vec v(r,t;\rho_1) \cdot \frac{-x}{|x|} = \frac{\partial}{\partial r} \rho_1^{m-1}(r) + \frac{\partial}{\partial r}(\rho_1*V)\geq 0.
\end{equation}

Let us assume that $V$ is given by (B); the argument for $V$ given by (A) is parallel and easier.  Recall that the stationary solution $\rho_A(x,t)$ satisfies the following equation in its positive set:
\begin{equation}
\frac{m}{m-1}\frac{\partial \rho_A^{m-1}}{\partial r} + \frac{\tilde{M}(r; \rho_A)}{\sigma_d r^{d-1}} = 0. \label{stat_ode}
\end{equation}
Therefore it follows that
\begin{eqnarray*}
\frac{m}{m-1}\frac{\partial}{\partial r} \rho_1^{m-1}(r) =a^{(m-1)d+1} \frac{m}{m-1}\frac{\partial \rho_A^{m-1}}{\partial r}(ar)= -a^{(m-1)d+1} \frac{\tilde{M}(ar; \rho_A)}{\sigma_d (ar)^{d-1}}.\end{eqnarray*}
Secondly observe that $\tilde M(r,t;\rho_1)$ satisfies
\begin{eqnarray*}
\tilde M(r,t;\rho_1) &=& \int_{B(0,r)} \int_{\mathbb{R}^d} a^d \rho_A(ay) \Delta V(y-x) dy dx\\
&=& \int_{B(0,ar)} \int_{\mathbb{R}^d} \rho_A(y) a^{-d} \Delta V(a^{-1}(y-x)) dy dx\\
&\geq & \int_{B(0,ar)} \rho_A * \Delta V dx \quad\text{(since $a^{-d}\Delta V(a^{-1}x) \succ \Delta V$ when $0<a<1$)}\\
&=& \tilde M(ar; \rho_A).
\end{eqnarray*} 
(Note that when $V$ is given by (A), direct computation yields $M(r,t;\rho_1) = M(ar; \rho_A)$.)

 Due to \eqref{nabla_rho*V_2} and the above inequalities, it follows that
\begin{eqnarray}
\nonumber
v(r,t; \rho_1)&=& \frac{\partial}{\partial r} \rho_1^{m-1}(r) + \frac{\partial}{\partial r}(\rho_1*V)\\
&\geq &  \frac{\partial}{\partial r} \rho_1^{m-1}(r) + \frac{\tilde{M}(r; \rho_1)}{\sigma_d r^{d-1}}\\
& \geq& (1-a^{d(m-2+2/d)}) a^{d-1}\frac{\tilde M(ar; \rho_A)}{\sigma_d (ar)^{d-1}}. \label{velocity_field_rescaled}
\end{eqnarray}
Since $m>2-2/d$,  the above inequality yields that the inward velocity field  $v(r, \rho_1) \geq 0$ when $a<1$. Therefore Corollary~\ref{velocity_field_corollary} implies that $ \rho(\cdot, t) \succ \rho_1$ for all $t\geq 0$. Since $\rho$ and $\rho_1$  have the same mass $A$, it follows that
$$
\{\rho(\cdot, t)>0\} \subset \{\rho_1(\cdot,t)>0\} \hbox{  for all } t>0,
$$
and we can conclude.

2. The assumption $\rho(0,0)>0$ can indeed be removed, since $\rho(0,t)$ would still become positive in finite time even if $\rho(0,0)=0$. This is because, for \eqref{pme}, it is a well-known fact that the solution will have a positive center density after finite time: this can be verified, for example, by maximum-principle type arguments using translations of Barenblatt solutions.

Note that a solution of \eqref{pme}  is a subsolution in the mass comparison sense. Hence one can compare $\rho$ with a solution $\psi$ of \eqref{pme} with initial data $\rho_0$ and apply Proposition \ref{comp_concentration} to conclude that $\psi \prec \rho$. Now our assertion follows due to the continuity of $\psi$ and $\rho$ at the origin. 
\end{proof}

\subsection{Exponential convergence towards the stationary solution in the subcritical regime}

As an application of  Proposition~\ref{comp_concentration}, we will prove the asymptotic convergence of radial solutions  to the unique radial stationary solution when the potential is given by (A) or (B).

\begin{theorem}[Exponential convergence of radial solutions with the Newtonian potential]
Let $m> 2-\frac{2}{d}$ and let $V$ be given by (A) or (B).  For given  $\rho_0\geq 0$: a continuous, radially symmetric function with compact support,  let $\rho(x,t)$ be the solution to \eqref{pde} with initial data $\rho_0$. Next let $\rho_A(x)$ be a radial stationary solution with mass $A:=\int \rho_0(x) dx$.  Then $M(r,t):= M(r,t;\rho) $ satisfies
$$
|M(r,t) - M(r; \rho_A)| \leq C_1 e^{-\lambda t},
$$
where $C_1$ depends on $\rho_0, A, m, d, V$, and the rate $\lambda$ only depends on $A, m, d, V$.
 \label{exp_conv}
\end{theorem}

\begin{proof} 

1. We will only prove the case when $V$ satisfies (B); the case for (A) can be proven with a parallel (and easier) argument. Note that we may assume $\rho_0(0)>0$ since otherwise $\rho(0,t)$ will become positive in finite time as explained in step 3 of the proof of Corollary~\ref{stay_in_compact_set}.

2. Let $\rho_A$ be a stationary solution with the same mass as $\rho_0$, given as in the proof of Corollary \ref{stay_in_compact_set}. Since $\rho_0$ is compactly supported, continuous and with $\rho_0(0)>0$, we can find a sufficiently small constant $a>0$ such that
$$
a^d{\rho_A}(ax)\prec \rho_0\quad\hbox{ and } a^{-d}{\rho_A}(a^{-1} x) \succ \rho_0.
$$

3. With the above choice of $a$, we next construct a self-similar subsolution $\phi(x,t)$ of \eqref{pde_for_m}
with initial data $\phi(x,0) = a^d{\rho_A}(ax)$ such that $M_\phi(\cdot, t) := M(\cdot,t;\phi)$ converges exponentially to $M(\cdot; \rho_A)$ as $t\to\infty$.

Here is the strategy of construction of $\phi(x,t)$ . Due to \eqref{velocity_field_rescaled}, for all $0<a<1$, the inward velocity field $v(r):=v(r;a^d\rho_A(ax))$ given by \eqref{inward_velocity} satisfies 
$$
v(r) \geq (1-a^{d(m-2+2/d)}) a^{d}r\frac{\tilde M(ar; \rho_A)}{\sigma_d (ar)^{d}} \geq 0.
$$
Observe that $\cfrac{d\tilde M(ar; \rho_A)}{\sigma_d (ar)^{d}}$ equals the average of $\rho_A*\Delta V$ in the ball $\{|x|\leq ar\}$.
 By Proposition~\ref{stat_sol_regularity_decreasing}, $\rho_A$ (hence $\rho_A*\Delta V$) is radially decreasing, and thus we have
\begin{equation}\label{bound_for_M}
C_1 \leq \frac{\tilde M(ar; \rho_A)}{\sigma_d (ar)^{d}} \leq C_2 \quad \text{in }\{\rho_A>0\},
\end{equation}
where $C_1$, $C_2$ only depend on $A, d, m, V$. This gives a lower bound for the inward velocity field $v$
\begin{equation}\label{ineq_v}
 v(r) \geq C_1 a^d(1-a^{d(m-2+2/d)})r.
 \end{equation}
We will use the above estimate to construct a subsolution $\phi(r,t)$ of \eqref{pde_for_m}. Let us define
\begin{equation}
\phi(r,t) = k^d(t)~{\rho_A}\big(k(t)r\big), \label{def_phi}
\end{equation} where the scaling factor $k(t)$ solves the following ODE
with initial data $k(0)=a$:
\begin{equation}
k'(t) = C_1 (k(t))^{d+1}(1-(k(t))^{d(m-2+2/d)}).\label{ode}
\end{equation}
Since $m>2-2/d$,  $k'(t) > 0$ when $0<k<1$, and since $k=1$ is the only non-zero stationary point for the ODE \eqref{ode}, for $0<k(0)<1$ we have $\lim_{t\to\infty} k(t)=1$.  Since 
$$
C_1 k^{d}(1-k^{d(m-2+2/d)}) = -C_1 d(m-2+2/d) (1-k) + o(1-k),
$$ it follows that
\begin{equation}\label{estimate_k}
0 \leq 1-k(t) \lesssim e^{-C_1 d(m-2+2/d)t},
\end{equation}
which implies
\begin{equation} \label{estimate_m_phi}
0 \leq M(r;\rho_A) - M_\phi(r,t) \lesssim e^{-C_1 d(m-2+2/d)t}.
\end{equation}

 Next we claim that $\phi$ is a subsolution of \eqref{pde_for_m}, i.e.,
\begin{equation}
\frac{\partial M_{\phi}}{\partial t} \leq \sigma_d r^{d-1}\frac{\partial}{\partial r} \Big((\frac{\partial M_{\phi}}{\partial r} \frac{1}{\sigma_d r^{d-1}})^m\Big)  + \Big(\frac{\partial M_{\phi}}{\partial r} \frac{1}{\sigma_d r^{d-1}}\Big) \tilde M_{\phi}  \quad\hbox{ in } \{\phi>0\}.\label{ineq_m_phi}
\end{equation}
To prove the claim,  first note that by definition of $\phi(r,t)$  we have $M_{\phi}(r,t) = {M}(k(t)r; \rho_A)$.
Hence, due to \eqref{ode} and definition of $\phi$,  the left hand side of \eqref{ineq_m_phi} can be written as 
\begin{eqnarray}
\nonumber\frac{\partial M_\phi}{\partial t}(r,t) &=& \partial_r M(k(t)r;\rho_A)~k'(t)r\\
&=& \sigma_d r^d{\rho_A}(k(t)r) k^{d-1}(t) k'(t)=\sigma_d r^d \phi(r,t) C_1 k^d(1-k^{d(m-2+2/d)}).
\end{eqnarray}

On the other hand, we can proceed in the same way as \eqref{ineq_v}, replacing $a$ by $k$, to obtain
$$
\frac{m}{m-1} \frac{\partial}{\partial r} \phi^{m-1} + \frac{\tilde M_\phi}{\sigma_d r^{d-1}} \geq C_1 k^d(1-k^{d(m-2+2/d)})r.
$$
Therefore
\begin{eqnarray*}
\nonumber\text{RHS of \eqref{ineq_m_phi}} \nonumber&=& \sigma_d r^{d-1}\frac{\partial}{\partial r} \phi^m  + \phi \tilde M_{\phi} \\
\nonumber&=& \sigma_d r^{d-1} \phi\big( \frac{m}{m-1} \frac{\partial}{\partial r} \phi^{m-1} + \frac{\tilde M_\phi}{\sigma_d r^{d-1}}\big)\\
&\geq& \sigma_d r^{d}
\phi C_1 k^d(1-k^{d(m-2+2/d)}),\end{eqnarray*}
thus $M_\phi$ indeed satisfies \eqref{ineq_m_phi}, and the claim is proved.  

4. Similarly one can construct a supersolution of \eqref{pde_for_m}. Let us define
\begin{equation*}
\eta(r,t) := k^d(t)~{\rho_A}\big(k(t)r\big), \label{def_eta}
\end{equation*} where $k(t)$ solves the following ODE
with initial data $k(0)=\frac{1}{a}$:
\begin{equation*}
k'(t) = C_2 k^{d+1}(1-k^{d(m-2+2/d)}),
\end{equation*}
where $C_2$ is defined in \eqref{bound_for_M}. Arguing parallel to those in as in step 3 yields that
 $\eta$ is a supersolution of \eqref{pde_for_m} and
\begin{equation}\label{estimate_m_eta}
0 \leq M_\eta(r, t) - M(r; \rho_A) \lesssim e^{-C_2 d(m-2+2/d)t}, \hbox{ for all } r >0.
\end{equation}

5. Lastly we compare $\phi, \eta$ with the weak solution $\rho$ of \eqref{pde}. Since
$$\phi(\cdot, 0) \prec \rho(\cdot, 0) \prec \eta(\cdot, 0) \quad\text{(see Figure \ref{initial_data_compare})},$$
\begin{figure}[h]
\centerline{
\epsfysize=1.8in
\epsfbox{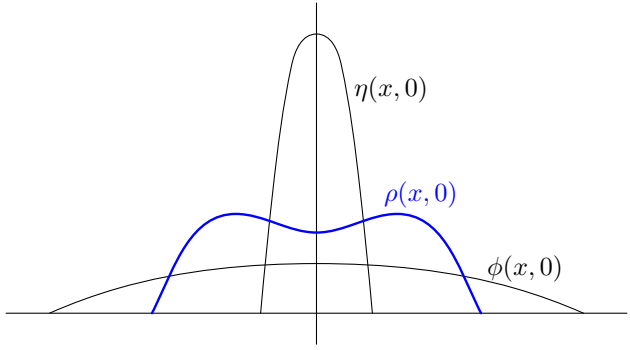}
}
\vspace{-0.3cm}
\caption{Initial data for $\phi$, $\rho$ and $\eta$}
\label{initial_data_compare}
\end{figure}

Proposition~\ref{comp_concentration} yields that
\begin{equation}
M_\phi(\cdot, t) \leq M(\cdot, t) \leq M_\eta(\cdot ,t). \label{compare_M}
\end{equation}
By \eqref{estimate_m_phi} and \eqref{estimate_m_eta}, we obtain
$$|M(r, t) - M(r; \rho_A)| \lesssim e^{-C_1 d(m-2+2/d)t} \hbox{ for } r\geq 0.$$
\end{proof}

Using the explicit subsolution and supersolution constructed in the proof of Theorem \ref{exp_conv}, we get exponential convergence of $\rho/A$ towards $\rho_A/A$ in the $p$-Wasserstein metric, which is defined below. Note that the Wasserstein metric is natural for this problem, since as pointed out in \cite{ags} and \cite{cmv}, the equation \eqref{pde} is a gradient flow of the 
energy \eqref{energy} with respect to the 2-Wasserstein metric. 
\begin{definition} \label{def_wasserstein}
Let $\mu_1$ and $\mu_2$ be two (Borel) probability measure on $\mathbb{R}^d$ with finite $p$-th moment. Then the \emph{$p$-Wasserstein distance} between $\mu_1$ and $\mu_2$ is defined as
$$
W_p(\mu_1,\mu_2) := \Big(\inf_{\pi\in \mathcal{P}(\mu_1, \mu_2)} \Big\{ \int_{\mathbb{R}^d\times\mathbb{R}^d} |x-y|^p \pi(dxdy)\Big\} \Big)^\frac{1}{p},$$
where $\mathcal{P}(\mu_1, \mu_2)$ is the set of all probability measures on $\mathbb{R}^d\times \mathbb{R}^d$ with first marginal $\mu_1$ and second marginal $\mu_2$.\end{definition}
\begin{corollary}\label{wasserstein} 
Let $\rho, \rho_A, A, C_1, \lambda$ be as given in Theorem \ref{exp_conv}. Then for all $p>1$, we have
$$W_p\Big(\frac{\rho(\cdot,t)}{A}, \frac{\rho_A}{A}\Big) \leq C_1 e^{-\lambda t}.$$
\end{corollary}
\begin{proof}
Before proving the corollary, we state some properties for Wasserstein distance, which can be found in \cite{vi}.  For two probability densities $f_0$, $f_1$ on $\mathbb{R}^d$, the $p$-Wasserstein distance between them coincides with the solution of Monge's optimal mass transportation problem.  Namely,
\begin{equation}W_p(f_1, f_0) = \Big( \inf_{T\#f_0 = f_1} \int_{\mathbb{R}^d} f_0(x)|x-T(x)|^p dx\Big)^{\frac{1}{p}}, \label{monge}
\end{equation}
where $T$ is a map from $\mathbb{R}^d$ to $\mathbb{R}^d$, and $T\#f_0 = f_1$ stands for ``the map $T$ transports $f_0$ onto $f_1$'', in the sense that for all bounded continuous function $h$ on $\mathbb{R}^d$,
$$\int_{\mathbb{R}^d} h(x) f_1(x) dx = \int_{\mathbb{R}^d} h(T(x)) f_0(x) dx.$$

Let $\phi$ be the subsolution constructed in the proof of Theorem \ref{exp_conv}, where we proved the radius of $\text{supp} ~\phi$ converges to the radius of $\text{supp} ~\rho_A$ exponentially in time. Note that $\phi(\cdot, t)$ is a rescaling of $\rho_A$, hence the convergence of support implies that there is a map $T_\phi(\cdot, t)$ transporting $\phi(\cdot, t)$ onto $\rho_A$ with $\sup_{x\in\{\rho(\cdot,t)>0\}}|x-T_{\phi}(x)|$ decaying exponentially in time.  Once we find such $T_\phi$, we can use \eqref{monge} to show that $W_p(\phi(\cdot, t), \rho_A)$ decays exponentially.

Without loss of generality, we assume the mass $A=1$ for the rest of the proof to avoid dividing $\rho$ and $\rho_A$ by $A$ every time. The transport map $T_\phi$ can be explicitly constructed as
$$T_\phi(x,t) = \frac{x}{|x|} M_\phi^{-1}\big(M(|x|;\rho_A), t\big).$$
Recall that $\phi$ is a rescaling of $\rho_A$, defined as $\phi(x, t) = (k(t))^d \rho_A(k(t)x)$, where $k(t)<1$ for all $t$, and $k(t)$ converges exponentially to 1. In this case the $T_\phi$ defined above can be greatly simplified as $T_\phi(x,t) = x/k(t)$. That gives us the following upper bound bound for $W_p(\phi(\cdot, t), \rho_A)$:
$$W_p(\phi(\cdot, t), \rho_A) \leq \Big(\int_{\mathbb{R}^d} \rho_A|x-\frac{x}{k(t)}|^p dx\Big)^{\frac{1}{p}} \leq R (1-\frac{1}{k(t)}),$$
where $R$ is the radius of support of $\rho_A$. Due to the estimate of $k(t)$ in \eqref{estimate_k}, we obtain the exponential decay
$$W_p(\phi(\cdot, t), \rho_A) \leq C_1 e^{-\lambda t}.$$
  We can apply the same argument to the supersolution $\eta(\cdot, t)$ as well.
  
To show that $W_p(\rho(\cdot, t), \rho_A)$ decays with the same rate, it is natural to consider the following map $T(\cdot, t)$ which transports $\rho_A$ onto $\rho(\cdot, t)$:
$$T(x,t) = \frac{x}{|x|} M^{-1}(M(|x|;\rho_A), t).$$
Then we have $|T(x,t) - x| = |M^{-1}(M(r;\rho_A), t) - r|$, where $r=|x|$.
Due to \eqref{compare_M}, we have 
$$M_\phi^{-1} \geq M^{-1}\geq M_\eta^{-1},$$ which gives
$$|T(x,t) - x| \leq \max\{|T_\phi(x,t)-x|, |T_\eta(x,t)-x|\}.$$
Hence we conclude that
$$W_p(\rho(\cdot, t), \rho_A) \leq W_p(\phi(\cdot, t), \rho_A)+W_p(\eta(\cdot, t), \rho_A) \leq C_1 e^{-\lambda t}.$$
\end{proof}
In fact one can also obtain the uniform convergence of $\rho(\cdot,t)$ to $\rho_A$ in sup-norm, however the convergence rate would depend on the modulus of continuity of $\rho$.
Theorem \ref{exp_conv} and the uniform continuity of $\rho$ and $\rho_A$, as well as the fact that $\rho_A$ is compactly supported,  yield the following: 
\begin{corollary}\label{cor:convergence}
Let $\rho$,$\rho_A$, $C_1$ and $\lambda$ be as given in Theorem~\ref{exp_conv}. Then we have
$$
\lim_{t\to\infty} \|\rho(x,t) - \rho_A(x)\|_{L^\infty(\R^d)} =0.
$$
\end{corollary}

Note that uniqueness of $\rho_A$ is not required in the proof of Theorem \ref{exp_conv}. Indeed, uniqueness of $\rho_A$ follows from the asymptotic convergence of $\rho$: if there are two radial stationary solutions $\rho_A^1$ and $\rho_A^2$ with the same mass, Corollary \ref{cor:convergence} implies $\rho(\cdot, t) \to \rho_A^i$ in $L^\infty$ norm for $i=1,2$ when $\rho$ is given as in Theorem \ref{exp_conv}. This 
immediately establishes the uniqueness of the radial stationary solution.
\begin{corollary} \label{cor:uniqueness}
Let V be given by (A) or (B), and let $m>2-\frac{2}{d}$. Then for all $A>0$, the radial stationary solution $\rho_A$ for \eqref{pde} with $\int \rho_A(x) dx=A$ is unique.
\end{corollary}

\subsection{Algebraic convergence towards Barenblatt profile in the supercritical regime}
In this section, we consider the asymptotic behavior of radial solutions in the supercritical regime, i.e. for $1<m< 2-\frac{2}{d}$. In this case the diffusion overrides the aggregation and thus the solution is expected to behave similar to that of Porous Medium Equation (PME) in the long run. In fact
recently it was shown in \cite{b1} (and also in \cite{s1}), by making use of entropy method as well as functional inequalities, that the solution of (1.1) with a general class of $V$  and with small mass and small $L^{(2-m)d/2}$ norm converges to the self-similar Barenblatt solution $\mathcal{U}(x,t)$ with algebraic rate,
\begin{equation}\label{barenblatt}
\mathcal{U}(x,t) = t^{-\beta d} (C - \frac{(m-1)\beta}{2m}|x|^2 t^{-2\beta})_+^{\frac{1}{m-1}},
\end{equation}
where $C$ is some constant such that $\|\mathcal{U}(\cdot, 0)\|_1 = \|\rho(\cdot, 0)\|_1$.

Here we will give a complementary result to \cite{b1} and \cite{s1} in the case of radial solutions, by using mass comparison  (Proposition~\ref{comp_concentration}). We point out that in our result the mass does not need to be small, and we provide an explicit description of solutions which are ``sufficiently scattered" so that they do not blow up in finite time.  Of course the method presented in \cite{b1} is much more delicate and yields optimal convergence results for general solutions with small mass in the supercritical regime.

Let $\rho$ be the weak solution to \eqref{pde}. Following \cite{v}, we re-scale $\rho$ as follows:
\begin{equation}\label{def_mu}
\mu(\lambda, \tau) = (t+1)^\alpha \rho(x,t);\quad \lambda = x(t+1)^{-\beta};\quad \tau=\ln (t+1),
\end{equation}
where
\begin{equation*}
\alpha=\frac{d}{d(m-1)+2}, ~~\beta = \alpha/d.
\end{equation*}
Then $\mu(\lambda, 0) = \rho(x,0)$, and $\mu(\lambda,\tau)$ is a weak solution of 
\begin{equation}
\mu_\tau = \Delta \mu^m + \beta \nabla\cdot(\mu \nabla\frac{|\lambda|^2}{2}) + e^{(1-\alpha)\tau} \nabla \cdot(\mu \nabla(\mu*(\mathcal{N}*\tilde{h}(\lambda, \tau))), \label{after_scaling}
\end{equation}
where
\begin{equation}\label{def_tilde_h}
\tilde{h}(\lambda, \tau) := e^{d\beta \tau}\Delta V(\lambda e^{\beta \tau}).
\end{equation}
(When $V=\mathcal{N}$ one should replace the last term by $e^{(1-\alpha)\tau} \nabla \cdot(\mu \nabla(\mu*\mathcal{N})$.)

In the absense of the last term, equation \eqref{after_scaling} is a Fokker-Plank equation
\begin{equation}
\mu_\tau = \Delta \mu^m + \beta \nabla\cdot(\mu \nabla\frac{|\lambda|^2}{2}) , \label{fokker_planck}
\end{equation}
 which is known to converge to the stationary solution $\mu_A$  exponentially, where $\mu_A$ has the mass $A := \| \mu(\cdot, 0)\|_1$ and satisfies
 \begin{equation}
 \frac{m}{m-1}\mu_A^{m-1} = (C-\beta \frac{|\lambda|^2}{2})_+ \hbox{ for some } C>0 . \label{stat_sol_as_tau_to_infty}
 \end{equation}

In Therorem \ref{exp_conv_m<2}, we will prove for $m< 2-2/d$, if the initial data is sufficiently less concentrated than $\mu_A$, then $\mu(\cdot,\tau)$ also converges to the same limit $\mu_A$ exponentially as $\tau\to\infty$. We begin by defining the following mass functions: 
$$
M^\mu(r, \tau) :={M}(r,\tau;\mu) \hbox{ and }\mathcal{\tilde M}(r, \tau; f) := \int_{B(0,r)} f * \tilde{h}(\cdot, \tau) d\lambda, 
$$ where ${M}$ is as given in \eqref{mass}, $\tilde h$ is as given in \eqref{def_tilde_h}, and $f$ is an arbitrary function.  Note that for $V=\mathcal{N}$, $\tilde{h}(\cdot, \tau)$ is the delta function for all $\tau$, hence $\mathcal{\tilde M} \equiv M$.

 Then $M^\mu$ satisfies the following PDE in the positive set of $\mu$:
\begin{equation} 
M^\mu_\tau = \sigma_d r^{d-1} (\frac{\partial M^\mu}{\partial r} \frac{1}{\sigma_d r^{d-1}}) \Big[\frac{m}{m-1}\frac{\partial}{\partial r}( (\frac{\partial M^\mu}{\partial r} \frac{1}{\sigma_d r^{d-1}})^{m-1}) + \beta r+ e^{(1-\alpha)\tau}\frac{\mathcal{\tilde M}(r,\tau;\mu)}{\sigma_d r^{d-1}} \Big]\label{pde_for_m_after_scaling}
\end{equation}
We first check that the mass comparison holds for re-scaled equations:
 \begin{proposition}  Let $V(x)$ be given by (A) or (B), and let $m< 2-\frac{2}{d}$. Assume $\mu_1(\lambda,\tau)$ is a subsolution and $\mu_2(\lambda,\tau)$ is a supersolution of \eqref{pde_for_m_after_scaling}. Further assume that
$\int \mu_1(\cdot, \tau) d\lambda$ and $\int \mu_2(\cdot, \tau) d\lambda$ stay constant for all $t\geq 0.$  Then the mass is ordered for all times, i.e.,
$$
\hbox{ if } \mu_1(\lambda,0) \prec \mu_2(\lambda,0), \hbox{ then we have } \mu_1(\lambda,\tau) \prec \mu_2(\lambda,\tau)\hbox{ for all } \tau>0.
$$ \label{comp_concentration_after_scaling}
\end{proposition}
\begin{proof}
Let $\rho_i(x,t)$ be the corresponding re-scaled versions of $\mu_i$. Then $\rho_1$ and $\rho_2$  are respectively a subsolution and a supersolution of \eqref{pde_for_m}. The proof then follows from  Proposition~\ref{comp_concentration} and from the fact that
$$
M(r,\tau; \mu_i) =e^{(\alpha-\beta)\tau} M(r e^{\beta\tau}, e^{\tau};\rho_i).
$$
\end{proof}
Next we state a technical lemma which is used later in the proof of the convergence theorem. The proof is in the appendix.
\begin{lemma}\label{technical}
Let $k(t)$ solve the ODE
\begin{equation}
k'(t) = C_1 k(1-k^{\alpha}) + C_2 k^{d+1}e^{-\beta t}, \label{ode_a}
\end{equation}
where $C_1, C_2, \alpha, \beta$ are positive constants.  Then there exists a constant $\delta>0$ such that if $0<k(0)<\delta$, then $k(t)\to 1$ exponentially as $t\to\infty$. \label{a_conv_exp}
\end{lemma}

Now we are ready to prove the main theorem. We will first prove it for radially decreasing solutions.

\begin{theorem} \label{exp_conv_m<2}
Let $V(x)$ be given by (A) or (B),  let  $1<m< 2-\frac{2}{d}$ and let $\mu_A$ be as given in \eqref{stat_sol_as_tau_to_infty}. Suppose $\mu_0(\lambda)$ is radially decreasing, compactly supported and has mass $A$. Then there exists a constant $\delta>0$ depending on $d, m, \mu_0$ and $V$, such that if  $$\mu_0(\lambda) \prec \delta^d \mu_A(\delta \lambda),$$  then the weak solution $\mu(\lambda, \tau)$ of \eqref{after_scaling} with initial data $\mu_0$ exists for all $\tau>0$. Furthermore,  $M(r,\tau;\mu)$ defined in \eqref{mass} converges to  $M(r,\tau;\mu_A)$ exponentially as $t\to\infty$ and uniformly in $r$.  \label{exp_conv_after_scaling}
\end{theorem}

\begin{proof}The proof of theorem is analogous to that of Theorem~\ref{exp_conv}: we will construct a self-similar subsolution $\phi(\lambda, \tau)$ and supersolution $\eta(\lambda, \tau)$ to \eqref{after_scaling}, both of which converge to $\mu_A$ exponentially.

Observe that \eqref{after_scaling} can be written as a transport equation
$$
\mu_t + \nabla \cdot (\mu \vec{v})=0,$$
where the \emph{velocity field} $\vec{v}$ is given by
$$\vec{v} :=  \frac{m}{m-1}\nabla( \mu^{m-1}) + \beta \lambda+ e^{(1-\alpha)\tau}\nabla(\mu*(\mathcal{N}*\tilde{h}(y, \tau)).$$
Hence the inward velocity field $v(r, \tau; \mu):= -\vec v \cdot \frac{x}{|x|}$ for the rescaled PDE \eqref{after_scaling} is
\begin{eqnarray*}
v(r, \tau; \mu) = \frac{m}{m-1}\frac{\partial}{\partial r}( \mu^{m-1}) + \beta r+ e^{(1-\alpha)\tau}\frac{\mathcal{ \tilde M}(r,\tau; \mu)}{\sigma_d r^{d-1}}.
\end{eqnarray*}
We first construct a  subsolution $\phi(\lambda, \tau)$ with the scaling factor $k(\tau)$ to be determined later:
$$\phi(\lambda, \tau) := k^d(\tau)\mu_A\big(k(\tau)\lambda\big).$$ 
Since $\mu_A$ satisfies \eqref{stat_sol_as_tau_to_infty}, the inward velocity field of $\phi$ is then given by
\begin{eqnarray*}
v(r,\tau; \phi) = (1-k^{d(m-1)+2})\beta r +  e^{(1-\alpha)\tau}\frac{\mathcal{\tilde M}(r, \tau;\phi)}{\sigma_d r^{d-1}}.
\end{eqnarray*}
Note that the last term of $v(r,\tau;\phi)$ is always non-negative, and thus $v(r,\tau; \phi) \geq (1-k^{d(m-1)+2})\beta r$. That motivates us to choose $k(\tau)$ to be the solution of the following equation
 \begin{equation}
k'(\tau) = \beta k(1-k^{d(m-1)+2}),
\end{equation}
with initial data $k(0)$ sufficiently small such that $\phi(\cdot, 0) \prec \mu_A$ and $\phi(\cdot, 0) \prec \mu(\cdot, 0)$. One can proceed as in the proof of Theorem~\ref{exp_conv} to verify $\phi$ is indeed a subsolution. Moreover, it can be easily checked that $k(\tau) \to 1$ exponentially as $\tau\to\infty$, hence $M(r,\tau;\phi)$ converges to  $M(r;\mu_A)$ exponentially as $\tau\to \infty$ and uniformly in $r$. 

Next we turn to the construction of a supersolution of the form 
$$
\eta(\lambda, \tau) := k^d(\tau)\mu_A\big(k(\tau)\lambda\big).
$$
 Here the main difficulty comes from the aggregation term, which might cause finite time blow-up of the solution.  To find an upper bound of the inward velocity field, we first need to control $\mathcal{\tilde M}(r,\tau, k^d \mu_A(k\lambda))$:
\begin{eqnarray*}
\mathcal{\tilde M}(r, \tau; k^d \mu_A(k\lambda)) &=& \int_{B(0,r)} k^d \mu_A(k \cdot) * e^{d\beta \tau} \Delta V(e^{\beta \tau} \cdot) (\lambda) d\lambda\\
&\leq& \|\Delta V\|_1 \int_{B(0,r)} k^d \mu_A(k\lambda) d \lambda\\
&=& \|\Delta V\|_1 \int_{B(0,kr)} \mu_A(\lambda) d \lambda~ \leq C(kr)^d/\sigma_d,
\end{eqnarray*}
where the first inequality is due to Riesz's rearrangement inequality and the fact that $\mu_A$ is radially decreasing.$C$ is some constant that does not depend on $k, r, \tau$.

The above inequality gives the following upper bound for the inward velocity field of $\eta$:
$$v(r,\tau; \eta) \leq (1-k^{d(m-1)+2})\beta r + C k^d e^{(1-\alpha)\tau} r.$$
Therefore if we let $k(t)$ solve the following ODE
\begin{equation}\label{ode_k}
k'(\tau) = \beta k(1-k^{d(m-1)+2}) + Ck^{d+1}  e^{(1-\alpha)\tau},
\end{equation}
and choose the initial data $k(0)$ such that $\mu(\cdot,0) \prec \eta(\cdot, 0) = k^d(0)\mu_A\big(k(0)\lambda\big) $, then $\eta$ would be a supersolution to  \eqref{after_scaling}.

  Let us choose $k(0) = \delta$, where $\delta$ is given in the assumption of this theorem. Due to Lemma \ref{a_conv_exp},  $k(\tau)\to 1$ exponentially as $\tau\to \infty$ when $\delta$ is sufficiently small, hence  it follows that $M(r,\tau;\eta)$ converges to  $M(r;\mu_A)$ exponentially.  

Since the supersolution $\eta$ exists globally, we claim that the weak solution $\mu$ exists globally as well.  Suppose not: then due to Theorem 4 of \cite{brb}, $\mu$ has a maximal time interval of existence $T^*$, and $\lim_{\tau\nearrow T^*} \| \mu(\cdot, \tau)\|_\infty = \infty$.  On the other hand, Proposition \ref{comp_concentration_after_scaling} yields that 
\begin{equation}
\mu(\cdot, \tau) \prec \eta(\cdot, \tau) \text{ for all }\tau<T^*.
\end{equation}  Note that Proposition \ref{decreasing} implies that $\mu$ is radially decreasing for all $\tau < T^*$,  which gives
\begin{equation}
\| \mu(\cdot, \tau)\|_\infty \leq \|\eta(\cdot, \tau)\|_\infty  \text{ for all }\tau<T^*.
\end{equation} 
The above inequality implies that $\lim_{\tau\nearrow T^*} \| \eta(\cdot, \tau)\|_\infty = \infty$, which contradicts the fact that $\|\eta(\cdot, \tau)\|_\infty$ is uniformly bounded for all $\tau$.

Once we have the global existence of $\mu$, Proposition \ref{comp_concentration_after_scaling} yields that
 $$\phi(\cdot, \tau) \prec \mu(\cdot, \tau) \prec 
\eta(\cdot, \tau) \text{ for all }\tau\geq 0. $$ Since both $\phi$ and $\eta$ converge exponentially towards $\mu_A$ as $\tau\to\infty$, we can conclude.
\end{proof}

The following generalization of Theorem~\ref{exp_conv_m<2} will be proved in Section 6. 

\begin{corollary}\label{exp_conv_m<2_corollary}
Let $V(x)$ be given by (A) or (B), and  $1<m< 2-\frac{2}{d}$. For a nonnegative function $\mu_0$ in $L^1(\mathbb{R}^d)$, define $A := \int \mu_0(\lambda) d\lambda$, and let $\mu_A(\lambda)$ be as given in \eqref{stat_sol_as_tau_to_infty}. Then the following holds:
\begin{itemize}
\item[(a)] there exists a small constant $\delta>0$ depending on $d, m, \mu_0$ and $V$, such that if
 $$\mu_0^*(\lambda) \prec \delta^d \mu_A(\delta \lambda),$$  then the weak solution $\mu(\lambda, \tau)$ of \eqref{after_scaling} with initial data $\mu_0$ exists for all $\tau>0$. 
 \item[(b)] Let $\mu_0$ is as given in (a) and also is radially symmetric and compactly supported, then $M(r,\tau;\mu)$ defined in \eqref{mass} converges to  $M(r,\tau;\mu_A)$ exponentially as $\tau\to\infty$ and uniformly in $r$.  
 \end{itemize}
 \label{exp_conv_after_scaling}
\end{corollary}

If we rescale back to the original space and time variables, Theorem \ref{exp_conv_m<2} immediately yields the algebraic convergence of mass function for the solution to \eqref{pde}.
\begin{corollary}\label{cor:above}
Let $V, m, \mu$ and $\mu_0$ be as given in Corollary \ref{exp_conv_m<2_corollary}, and let $\rho$ be given by \eqref{def_mu}.   Let $\mathcal{U}(x,t)$ be as given in \eqref{barenblatt}.  Then $\rho$ is a weak solution to \eqref{pde}, and $\rho$ vanishes to zero as $t\to\infty$ with algebraic decay.  In particular if $\rho_0$ is radially symmetric then 
\begin{itemize}
\item[(a)] $|M(r,t)- M(r,t; \mathcal{U})| \leq Ct^{-\gamma}, \text{ for all } r\geq 0,$ for some $C, \gamma$ depending on $\rho_0, m, d$ and $V$. 
\item[(b)] for all $p>1$ we have
$$W_p(\frac{\rho(\cdot,t)}{A}, \frac{\mathcal{U}(\cdot, t)}{A}) \leq C t^{-\gamma},$$
where $C, \gamma$ depend on $\rho(x,0), m, d$ and $V$.
\end{itemize}
\end{corollary}

\begin{proof}
From the proof of Theorem \ref{exp_conv_m<2}, we have $W_p(\frac{\mu(\cdot, \tau)}{A}, \frac{\mu_A}{A}) \lesssim e^{-\gamma \tau}$ for some $\gamma$ depending on $\rho(x,0)$, $m$, $d$ and $V$, and the proof is analogous with the proof of Corollary \ref{wasserstein}.  Now we scale back, and the above inequality becomes
$$W_p(\frac{\rho(, \tau)}{A}, \frac{\rho_A}{A}) \lesssim (t+1)^{-\gamma}.$$
\end{proof}

\section{A comparison principle for general solutions and Instant regularization in $L^\infty$}

In this section we consider general (non-radial) solutions of \eqref{pde}. Our goal is to prove the following result:

\begin{theorem}\label{rearrangement}
Suppose $m>1$. Let $V$ be given by (A) or (B), and let $\rho$ be the weak solution to \eqref{pde} with initial data $\rho(x,0) = \rho_0(x)$.  Let $\bar \rho$ be the weak solution to \eqref{pde} with initial data $\bar \rho(x,0) = \rho_0^*(x)$.  Assume $\bar \rho$ exists for $t\in[0,T)$, where $T$ may be infinite. Then $\rho^*(\cdot, t) \prec \bar\rho(\cdot, t)$ for all $t\in[0,T)$.
\end{theorem}

As an application of Theorem \ref{rearrangement}, we will show that solutions of \eqref{pde} with initial data in $L^1$ immediately regularize in $L^\infty$ (see Proposition~\ref{instant}.)

The proof of Theorem~\ref{rearrangement}, which we divide into several subsections follows that of the corresponding theorem for solutions of \eqref{pme} (see Chapter 10 of \cite{v}). The theorem in \cite{v} is proved by taking the semi-group approach and applying the Crandall-Liggett Theorem. The challenge lies in the fact that our operator in \eqref{pde} is not a contraction, in either $L^1$ or $L^\infty$. For this reason the proof requires an additional approximation of our equation with one with fixed drift: see \eqref{pmedrift2}.

\subsection{Implicit Time Discretization for PME with drift}
Consider the following equation
\begin{equation}
\rho_t = \Delta \rho^m + \nabla \cdot (\rho\nabla \Phi), \label{pmedrift}
\end{equation}
where $\Phi$ is a function given \emph{a priori} such that $\Phi(x,t) \in C(\mathbb{R}^d \times [0,\infty))$, and $\Phi(\cdot,t) \in C^2(\mathbb{R}^d)$ for all $t$.

Following the proof in the case of \eqref{pme} in \cite{v}, we approximate \eqref{pmedrift} via an implicit discrete-time scheme.  For a small constant $h>0$, $U_i$ is recursively defined as the solution of the following equation:
\begin{equation}
\frac{U_i - U_{i-1}}{h} = \Delta U_i^m + \nabla \cdot (U_i \nabla \Phi_i),  ~~i=1,2,\ldots\label{itd}
\end{equation}
where $U_0 = u(\cdot,0), \Phi_i = \Phi(\cdot, ih)$. Now define
\begin{equation}\label{def_rho_h}
\rho_h(\cdot, t) := U_i(\cdot) \quad \text{for }(i-1)h< t\leq ih,~~  i=1,2,\ldots
\end{equation}

The following result states that our approximation scheme is valid: the proof is in the Appendix.
\begin{proposition} \label{itd_conv}
Let $u_0\in L^1(\R^d;(1+|x|^2)dx)\cap L^\infty(\R^d)$, and let $\rho_h$ be defined by \eqref{def_rho_h}.Then  there exists a function $\rho\in L^\infty([0,\infty);L^1(\R^d))$ such that  
\begin{equation*}
\sup_{0\leq t\leq T}\|\rho(\cdot,t) - \rho_h(\cdot,t)\|_{L^1(\R^d)} \to 0
\end{equation*}
for any $T>0$. Moreover, $\rho$ coincides with the unique weak solution for \eqref{pmedrift}.
\end{proposition}

\subsection{Rearrangement comparison}

For a given function $u(x): \R^d\to \R$, let us define $u^*$ as given in \eqref{def:rearrangement}.

Consider the following equation, where $f(x,t)\in C([0,\infty); L^1(\mathbb{R}^d)$) is a given function:
\begin{equation}
\rho_t = \Delta \rho^m + \nabla \cdot (\rho\nabla (f*V)). \label{pmedrift2}
\end{equation}

\begin{theorem} \label{pmedrift_conti_comp}
Suppose $m>1$. Let $V$ be given by (A) or (B), and let $\rho$ be the weak solution to \eqref{pmedrift2} with initial data $\rho(x,0) = \rho_0(x)$. Let $\bar{\rho}$ be the weak solution to the symmetrized problem
\begin{equation}
\rho_t = \Delta \rho^m + \nabla \cdot (\rho\nabla (f^* *V)), \label{pmedrift_symm}
\end{equation}
with initial data $\bar{\rho}(x,0) = \rho_0^*(x)$. Then $\bar \rho$ is radially decreasing, and 
$$
\rho^*(\cdot,t) \prec \bar{\rho}(\cdot, t)\hbox{ for all } t>0.
$$
\end{theorem}

 Due to Proposition~\ref{itd_conv}, to prove Theorem~\ref{pmedrift_conti_comp} it suffices to show the following Proposition; see the Appendix for the proof. 

\begin{proposition}\label{rearrangement_comp_one_step}Suppose $V$ is given by (B) and $m>1$.
Let $u\in D$ (the domain D is defined in \eqref{d_def}) be the weak solution of
\begin{equation}\label{one-step}
-h\Delta u^m -h\nabla\cdot(u \nabla(f*V))+u = g,
\end{equation}
where $f, g\in L^1(\mathbb{R}^d)$ are nonnegative. Also, let $\bar{u}\in D$ be the solution to the symmetrized problem, i.e. $\bar{u}$ solves \eqref{one-step} with $f,g$ replaced by $\bar{f}$ and $\bar g$ respectively, where $\bar f$ and $\bar g$ are radially decreasing, have the same mass as $f$ and $g$ respectively, and satisfy $f^* \prec \bar f$ and  $g^* \prec \bar g$.  Then  $u^*\prec \bar{u}$.
\end{proposition}

\textbf{Proof of Theorem \ref{pmedrift_conti_comp}:}

The radial monotonicity of $\bar \rho$ can be shown via a similar argument as in Theorem \ref{decreasing}: in fact the argument is easier here since $f^* * \Delta V$ is a radially decreasing function.

Next we prove $\rho^* \prec \bar \rho$ for all $t\geq 0$. Let $U_i$ be the discrete solution for the original problem, and let $V_i$ be the discrete solution for the symmetrized problem. Due to Proposition~\ref{itd_conv} it suffices to prove that $U_i^* \prec V_i$ for all $i\in \mathbb{N}$. Here $U_i$ solves
\begin{equation}
\frac{U_i - U_{i-1}}{h} = \Delta U_i^m + \nabla \cdot (U_i \nabla (f_i*V)),
\end{equation}
where $U_0 = u(\cdot,0), f_i = f(\cdot, ih)$, and $V_i$ solves
\begin{equation}
\frac{V_i - V_{i-1}}{h} = \Delta V_i^m + \nabla \cdot (V_i \nabla (f_i^* *V)),
\end{equation}
where $V_0 = u^*(\cdot,0)$. Since $U_0^* \prec V_0$, by applying Proposition~\ref{rearrangement_comp_one_step} inductively we can conclude. Lastly when $V=\mathcal{N}$, we can use a mollified Newtonian kernel to approximate $\mathcal{N}$, and the result follows via Proposition \ref{approximation}. \hfill$\Box$

Now we are ready to prove our main result:

\textbf{ Proof for Theorem~\ref{rearrangement}:}
Let us first prove the theorem when $V$ is given by (B), where we have global existence of solutions. Let $\rho_1(\cdot, t) := \rho^*(\cdot, t)$ for all $t\geq 0$, where $\rho(x,t)$ is the weak solution of \eqref{pde} with initial data $\rho(x,0) = \rho_0(x)$.  For $i>1$, we let $\rho_i$ be the weak solution to the following equation:
\begin{equation}
(\rho_i)_t = \Delta (\rho_i)^m + \nabla \cdot (\rho_i \cdot \nabla (\rho_{i-1} * V)),
\end{equation}
with initial data $\rho_i(x,0) = \rho^*(x,0)$. Observe that $\rho_i(\cdot, t)$ is radially decreasing for all $i\in \mathbb{N}^+, t\geq0$.

By Theorem \ref{pmedrift_conti_comp}, we have $\rho_i \prec \rho_{i+1}$ for all $i\in \mathbb{N}$.  Hence we have
\begin{equation}\label{sequence_comp}
\rho^*(\cdot, t) = \rho_1(\cdot, t) \prec \rho_2(\cdot, t) \prec \rho_3(\cdot, t) \prec \ldots, ~\hbox{ for all } t.
\end{equation}
Due to Theorem~\ref{unif_cont}, $\{\rho_i\}$ is locally uniformly continuous in space and time.
Hence by the Arzela-Ascoli Theorem any subsequence of $\{\rho_i\}$ locally uniformly converges to a function $\bar{\rho}$ along a subsequence. On the other hand $\bar\rho$ is the unique weak solution for \eqref{pde} with initial data $\bar\rho(x,0) = \rho_0^*(x)$. This means that the whole sequence $\{\rho_i\}$ locally uniformly converges to $\bar{\rho}$. Now we can conclude due to \eqref{sequence_comp}.

When $V=\mathcal{N}$, we can use a mollified Newtonian kernel to approximate $\mathcal{N}$, and the result follows via Proposition \ref{approximation}.  \hfill$\Box$
 \begin{corollary}\label{lp_compare}
Suppose $m>1$. Let $V$ be given by (A) or (B), and let $\rho$ be the weak solution of \eqref{pde} with initial data $\rho_0(x)$.  Let $\bar \rho$ be the solution to the symmetrized problem, i.e. $\bar \rho$ is the weak solution to \eqref{pde} with initial data $\rho_0^*(x)$.  Assume $\bar \rho$ exists for $t\in[0,T)$, where $T$ may be infinite. Then for any $p\in (1,\infty]$ we have
$$
\|\rho(\cdot, t)\|_{L^p(\mathbb{R}^d)} \leq \|\bar\rho(\cdot, t)\|_{L^p(\mathbb{R}^d)},~ \hbox{ for all } t\in [0,T).
$$
\end{corollary}

We are now ready to generalize Theorem \ref{exp_conv_m<2} to non-radial solutions.

\textbf{Proof of Corollary \ref{exp_conv_m<2_corollary}: }
Let $\bar \mu(\lambda, \tau)$ be the weak solution to \eqref{after_scaling} with initial data $\mu_0^*(\lambda)$.  Then $\bar \mu(\cdot, 0)$ meets the assumptions for Theorem \ref{exp_conv_m<2}, which implies the global existence of $\bar \mu$. Due to Corollary \ref{lp_compare}, $\|\mu(\cdot, \tau)\|_\infty \leq \|\bar \mu(\cdot, \tau)\|_\infty$ for all $\tau$ during the existence of $\mu$; hence the uniform boundedness of $\bar \mu$ yields that $\mu$ cannot blow up and thus must exist globally in time. This proves (a).

Now suppose $\mu_0$ is radially symmetric and compactly supported, and $\mu_0$ satisfies the assumption in (a) such that the corresponding solution $\mu$ exists globally in time. In this case we can construct subsolution and supersolution as in the proof for Theorem \ref{exp_conv_m<2} to prove (b).
\hfill$\Box$

\subsection{Instant regularization in $L^\infty$}

We finish this section by presenting the following regularization result as a corollary of Theorem~\ref{rearrangement}.  It says that for initial data $\rho_0 \in L^1(\mathbb{R}^d;(1+|x|^2)dx)\cap L^\infty(\mathbb{R}^d)$, no matter how large the $L^\infty$ norm of $\rho_0$ is, $\|\rho(\cdot, t)\|_\infty$ will always be bounded by $t^{-\alpha}$ for some short time. 
\begin{proposition}\label{instant}
Let $V$ be given by (A) or (B), and let $m>2-2/d$.  Let $\rho(x,t)$ be the weak solution for \eqref{pde}, with initial data $\rho_0 \in L^1(\mathbb{R}^d;(1+|x|^2)dx)\cap L^\infty(\mathbb{R}^d)$. Let us denote $A=\| \rho_0 \|_1$ and $\alpha := \frac{d}{d(m-1)+2}$. Then there exists $c=c(m,d,A,V)$ and $t_0=(2c)^{1/\alpha} >0$ such that we have $\rho(\cdot,t)\in L^\infty(\mathbb{R}^d)$ with
$$
\|\rho(\cdot,t)\|_{L^\infty(\mathbb{R}^d)} \leq c(m,d, A, V) t^{-\alpha}\hbox{ for all }0< t<t_0.
$$

\end{proposition}
\begin{proof} By Corollary \ref{lp_compare}, it suffices to prove the inequality when $\rho_0$ is radially symmetric.
Also, in this proof we denote $c(m,d, A, V)$ to be all constants which only depend on $m, d, A, V$.

 Let $\rho_A$ be the radial stationary solution of \eqref{pde} with mass $A$. Note that $\rho_A$ is radially decreasing, and thus $\rho_A(0)>0$.  Since $u_0$ is a radial function in $L^\infty$, we can scale $\rho_A$ to make it more concentrated than $u_0$, i.e. we choose $0<a<1$ to be sufficiently small such that
$$
u_0 \prec a^{-d}{\rho_A}(a^{-1} x).
$$
As in the proof of Theorem~\ref{exp_conv}, let us define
\begin{equation*}
\eta(r,t) := k^d(t)~{\rho_A}\big(k(t)r\big), \label{def_eta_2}
\end{equation*} where $k(t)$ solves the following ODE
with initial data $k(0)=a^{-1}$:
$$
k'(t) =  c(m,d, A, V) k^{d+1}(1-k^{d(m-2+2/d)}).
$$
Here $c(m,d, A, V)$ corresponds to $C_2$ in the proof for Theorem~\ref{exp_conv}.  It was shown in the proof that
 $$
 \rho(\cdot,t)\prec\eta(\cdot, t) \quad\hbox{ for all } t\geq 0,
 $$
which in particular yields that
 $$
 \|\eta(\cdot, t)\|_{L^\infty(\mathbb{R}^d)} \geq \|\rho(\cdot, t)\|_{L^\infty(\mathbb{R}^d)}\quad \hbox{ for all } t\geq 0.
 $$
Observe that, by definition,
$$
h(t) := \|\eta(\cdot, t)\|_{L^\infty(\mathbb{R}^d)} = k^d(t) \rho_A(0) = c(m,d, A, V) k^d(t).
$$

Therefore to prove our proposition it is enough to show
$$
h(t) \leq f(t):=c(m,d, A,V) t^{-\alpha} \quad\text{for all } h(0)>0 \text{ and } t\in[0,t_0],
$$
where $t_0$ is chosen such that $f(t) \geq 2$.
Note that  $h(t)$ solves
\begin{eqnarray*}
h'(t) &=& c(m,d, A,V) k^{d-1}k' \\
&=& c(m,d, A,V) h^2 \big(1- h^{m-2+2/d}\big).
\end{eqnarray*}
In particular when $h(t) \geq 2$, $h$ satisfies the following inequality
$$
h'(t) \leq -c(m,d, A, V) h^{m+2/d}.
$$
Since $f(t)$ solves the above ODE with equality, we obtain $h(t) \leq f(t)$ for $0\leq t\leq t_0$. Now we are done.
\end{proof}

\appendix
\section{Appendix}

\subsection{Proof of existence for $\rho$ as given in Proposition 2.1}

Here we will show the existence of a global minimizer for the free energy functional given in \eqref{energy}.  First note that the kernels $V$ given in (A) and (B) belong to $M^{\frac{d}{d-2}}$, where $M^p$ denotes the weak $L^p$ space.

Our proof is based on a theorem of Lions in \cite{l}:
\begin{theorem}[\cite{l}] \label{lions_thm}
Suppose $f\in M^p(\mathbb{R}^d)$, $f\geq 0$ and consider the problem 
\begin{equation*}
I_\lambda = \inf_{u\in K_\lambda} \Big\{ \int_{\mathbb{R}^d} \frac{1}{m-1}u^m dx - \frac{1}{2} \int_{\mathbb{R}^d} u(u*f) dx \Big\},  \label{I_lambda}
\end{equation*}
where
\begin{equation*}
K_\lambda = \Big\{ u\in L^q(\mathbb{R}^d) \cap L^1(\mathbb{R}^d), ~u\geq 0 \emph{ a.e., } \int_{\mathbb{R}^d} u dx = \lambda\Big\} \emph{ with } q=\frac{p+1}{p}.
\end{equation*}
Then there exists a minimizer of problem $(I_\lambda)$ if and only if the following  holds:
\begin{equation}
I_\lambda < I_\alpha + I_{\lambda-\alpha},~ \forall \alpha \in (0,\lambda). \label{condition_lions}
\end{equation}
\end{theorem}

\begin{proposition}[\cite{l}] \label{equiv_cond_lions}
Suppose there exists some $\alpha\in (0,d)$ such that $f(tx)\geq t^{-\alpha}f(x)$ for all $t\geq 1$.
Then (\ref{condition_lions}) holds if and only if 
\begin{equation}\label{I_lambda<0}
I_\lambda < 0, \hbox{ for all } \lambda>0.
\end{equation}
\end{proposition}

 For the rest of this subsection, we will verify that Proposition~\ref{equiv_cond_lions} applies to our kernels.
\begin{proposition}\label{above11}
Let $f=V$ given by either \eqref{kernelA} or \eqref{kernelB}. Then $f(tx)\geq t^{-\alpha} f(x)$ with $t\geq 1$ and $\alpha=d-2$.
\end{proposition}

\begin{proof}
When $V$ is given by (A) the proof is straightforward,  so suppose $V$ is given by (B). Then  $f=-\mathcal{N}*h$.  Since $h$ can be approximated by a sum of indicator functions, it suffices to prove the proposition for $f = -\mathcal{N}*\chi_{B(0,r)}$, where $\chi$ is the indicator function. In this case we have
\begin{equation} f(x) =
\begin{cases}
\cfrac{1}{2(d-2)}r^2 - \cfrac{1}{2d}|x|^2 & \text{ for }|x|\leq r,\\ 
\cfrac{1}{d(d-2)}r^d |x|^{-d+2} & \text{ for }|x|> r,
\end{cases}
\end{equation}
which finishes the proof.
\end{proof}

\begin{proposition}
Let $f$ be as in Lemma~\ref{above11}.  Suppose $m>2-\frac{2}{d}$ and let us define
$$
u = \frac{ \lambda \chi_{B(0,R)}}{c_d R^d}
$$
 where $c_d$ is the volume of the unit ball in $\mathbb{R}^d$ and $R$ is a constant to be chosen later.  If $R$ is sufficiently large, we have
$$E(u) := \int_{\mathbb{R}^d} \frac{1}{m-1}u^m dx - \frac{1}{2} \int_{\mathbb{R}^d} u(u*f) dx < 0,$$
and thus  $I_\lambda \leq E(u) < 0$.
\end{proposition}

\begin{proof}
 First note that we have
$$
\int_{\mathbb{R}^d} \frac{1}{m-1}u^m dx = \int_{\mathbb{R}^d} \frac{1}{m-1} \Big( \frac{ \lambda \chi_{B(0,R)}}{c R^d}\Big)^m dx  \simeq \lambda^mR^{-d(m-1)}.
$$
On the other hand, $\int_{B(0,R/2)} (-V) dx \simeq  R^2$ if $R$ is sufficiently large: this implies\\
$\int_{\mathbb{R}^d} u(u*(-V)) dx \gtrsim \lambda^2 R^{-d+2}.$
Since  $m>2-\frac{2}{d}$ we have 
$$
E(u) =\int_{\mathbb{R}^d} \frac{1}{m-1}u^m dx -  \int_{\mathbb{R}^d} u(u*(-V)) dx<0.
$$
\end{proof}
\subsection{Proof of Lemma~\ref{integral_lemma}}

Observe that $\Delta V$ is nonnegative and radially decreasing, and thus it can be approximated in $L^1(\mathbb{R}^d)\cap L^\infty(\mathbb{R}^d) $ by the sum of bump functions of the form $c\chi_{B(0,r)}$, where $c>0$. By linearity of convolution, it suffices to prove that for each bump function $\chi_{B(0,r)}$, where $r$ is any positive real number, we have
\begin{equation}
(u*\chi_{B(0,r)})(b_1) - (u*\chi_{B(0,r)})(a_1) \leq \|\chi_{B(0,r)}\|_1 (u(b_1)-u(a_1)).\label{bump_u}
\end{equation}
Observe that
\begin{eqnarray}
 (u*\chi_{B(0,r)})(b_1) - (u*\chi_{B(0,r)})(a_1) &=& \int_{B(b_1,r)}u(x) dx - \int_{B(a_1,r)}u(x) dx\\
&=& \int_{\Omega_B} u(x) dx - \int_{\Omega_A} u(x)dx\label{difference_of_u}.
\end{eqnarray}

Here $\Omega_A := B(a_1, r) \backslash B(b_1, r)$ and $\Omega_B := B(b_1, r) \backslash B(a_1, r)$ (see Figure \ref{omega}).

\begin{figure}[h]\label{omega_A_B}
\begin{center}
\centerline{
\epsfysize=2in
\epsfbox{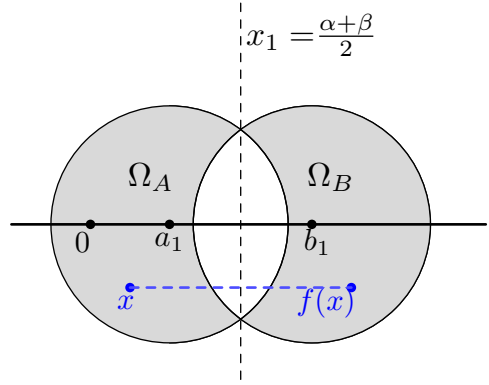}
}\end{center}
\vspace{-1cm}
\caption{The domains $\Omega_A$ and $\Omega_B$}
\label{omega}
\end{figure}

 Note that $\Omega_A$ and $\Omega_B$ are symmetric about the hyperplane $H=\{x: x_1 = \frac{\alpha+\beta}{2}\}$. For any $x\in\Omega_A$, use $f(x)$ to denote the reflection point of $x$ with respect to $H$. Then we have
$$
\int_{\Omega_B} u(x) dx - \int_{\Omega_A} u(x) dx = \int_{\Omega_A} (u(f(x)) - u(x)) dx.
$$
Since $|x|<|f(x)|$ for $x\in\Omega_A$,  we can use the assumption \eqref{assumption_ab} to obtain
$$\int_{\Omega_A} (u(f(x)) - u(x)) dx  \leq \int_{\Omega_A} (u(b)-u(a)) dx \leq |B(0,r)| (u(b)-u(a)),$$
which completes the proof.

\subsection{Proof of Lemma~\ref{technical}}
\begin{proof}
When $0<k<1$, the right hand side of \eqref{ode_a} is bounded above by $(C_1+C_2)k$. Hence if the initial data satisfies $0<k(0)<1$, the inequality $k(t) \leq k(0) e^{(C_1+C_2)t}$ will hold until $k$ reaches 1.  In other words, $k(t)$ is guaranteed to be smaller than 1 until time $t_1 := -\frac{\ln k(0)}{C_1+C_2}$.

Now if we choose $k(0)$ to be sufficiently small such that $0<k(0)<\delta$, where 
$$\delta := (\alpha C_1 C_2^{-1} 2^{-d-2})^{\frac{C_1+C_2}{\beta}},$$ then $t_1$ would be sufficiently large such that
$$C_2 2^{d+1} e^{-\beta t_1} \leq \frac{C_1 \alpha}{2}.$$
We claim $g(t) := 1+e^{-\epsilon (t-t_1)}$ is a supersolution of \eqref{ode_a} for $t \geq t_1$, where $\epsilon := \min\{\beta, \frac{1}{2}C_1\alpha\}$. It is obvious that $g(t_1) > 1 \geq k(t_1)$, so it suffices to show
\begin{equation}
g'(t) \geq C_1 g(1-g^{\alpha}) + C_2 g^{d+1}e^{-\beta t} \quad \text{for }t\geq t_1. \label{ode_g}
\end{equation}
By definition of $g$, we have
\begin{eqnarray}
\text{RHS of  \eqref{ode_g}} &\leq& -C_1 \alpha e^{-\epsilon(t-t_1)} + C_2 2^{d+1} e^{-\beta t_1} e^{-\beta (t-t_1)}\\
&\leq& -\frac{1}{2} C_1 \alpha e^{-\epsilon (t-t_1)}\\
&\leq&  -\epsilon e^{-\epsilon(t-t_1)} ~=\text{LHS of  \eqref{ode_g}} . 
\end{eqnarray}
Therefore $k(t) \leq 1+e^{-\epsilon(t-t_1)}$ for all $t\geq t_1$. 

To obtain the corresponding lower bound for $k(t)$, note that the last term of \eqref{ode_a} is non-negative. Therefore if $g$ solves $g'(t) = C_1 g(1-g^\alpha)$ and $g(0)=k(0)$, then $|g(t) -1| \lesssim e^{-C_1\alpha t}$. Comparison between these two ODEs yields $k(t) \geq g(t)$ for all $t \geq 0$, which implies $k(t) \geq 1-Ce^{-C_1 \alpha t}$.  Now we can conclude that there exists $C$ depending on $C_1, C_2, \alpha,\beta$ and $k(0)$ such that 
$$
|k(t)-1| \leq C e^{-\epsilon t}.
$$
\end{proof}
\subsection{Proof of Proposition~\ref{itd_conv}}
The proof of Proposition~\ref{itd_conv} is an application of the Crandall-Liggett Theorem (\cite{cl}, also see Theorem 10.16 in \cite{v}). Let us consider the following domain:
\begin{equation} \label{d_def}
D: = \Big\{ u\in L^1(\mathbb{R}^d): u^m \in W^{1,1}_{\text{loc}}(\mathbb{R}^d), \Delta u^m \in L^1(\mathbb{R}^d),  |\nabla u^m| \in M^{d/(d-1)}(\mathbb{R}^d)\Big\}.
\end{equation}

Here the Marcinkiewicz space $M^p(\mathbb{R}^d), 1<p<\infty$, is defined as the set of $f\in L^1_{loc}(\mathbb{R}^d)$ such that
$$\int_{K} |f(x)| dx \leq C|K|^{(p-1)/p},$$
for all subsets K of finite measure.  The minimal $C$ in the above inequality gives a norm in this space, i.e.
$$\|f\|_{M^p(\mathbb{R}^d)} = \sup \Big \{ \text{meas}(K)^{-(p-1)/p} \int_K |f| dx: K\subset \mathbb{R}^d, \text{meas}(K)>0 \Big \}.$$

A parallel argument as in Theorem 2.1 of \cite{bbc} yields the existence of solutions for the discretized equation.
\begin{lemma}[Existence]\label{existence}
Let $d\geq 3$ and let $u_0\in L^1(\R^d), \Phi \in C^2(\mathbb{R}^d)$. Then there exists a unique weak solution $u\in D$ of the following equation:
\begin{equation}\label{pmedrift_onestep}
\frac{u-u_0}{h} = \Delta u^m + \nabla \cdot (u \nabla \Phi).
\end{equation}
\end{lemma}

The proof of the next lemma  is parallel to that of Prop 3.5 in \cite{v} for \eqref{pme}. 

\begin{lemma}[$L^1$ contraction] \label{contraction} Let $\Phi \in C^2(\mathbb{R}^d)$, $u_{0i}\in L^1(\R^d)$  and let $u_1, u_2 \in D$ be the weak solutions to the degenerate elliptic equation
\begin{equation}\label{elliptic_pmedrift}
\frac{u_i - u_{0i}}{h} = \Delta (u_i)^m + \nabla \cdot (u_i \nabla \Phi),~~ i= 1,2.
\end{equation}
Then  $u_1$ and $u_2$ satisfy
\begin{equation}
\|u_1 - u_2\|_{L^1(\mathbb{R}^d)} \leq \|u_{01} - u_{02}\|_{L^1(\mathbb{R}^d)}.
\end{equation}
\end{lemma}

\textbf{Proof of Proposition~\ref{itd_conv}}
\begin{proof}
Let $D$ be defined above, and define the nonlinear operator $\mathcal{A}: D \to L^1(\mathbb{R}^d)$  by 
$$\mathcal{A}(u) = -\Delta u^m - \nabla(u\nabla\Phi),$$
Then Lemma \ref{existence} and Lemma \ref{contraction} yield that for any $h>0$, there is a unique solution $u$ in $D$ solving
$$h\mathcal{A}(u) + u = f.$$
Moreover the map $f \mapsto u$ is a contraction in $L^1(\mathbb{R}^d)$. Now arguing as in \cite{v}, the Crandall-Liggett Theorem yields the conclusion.
\end{proof}

\subsection{Proof of Proposition~\ref{rearrangement_comp_one_step}}
The proof of Proposition~\ref{rearrangement_comp_one_step} is parallel to that of Theorem 11.7 in \cite{v} for \eqref{pme}.  First we state a lemma which deals with the extra convolution term.

\begin{lemma} \label{rearrangement_comp} Let $V$ be given by (B). Let $f\in L^1(\mathbb{R}^d)$ and $\phi\in W_0^{1,\infty}(\mathbb{R}^d) $ be non-negative functions. Then for any non-negative number $a,b$, we have
\begin{equation}
\int_{\{a<\phi<b\}} \nabla(f *(-V))\cdot \nabla \phi \leq \int_{\{\phi^*>a\}} (f^* * \Delta V) (\max\{\phi^*, b\} -a), \label{comp_integral}
\end{equation}
where the equality is achieved if $f, \phi$ are both radially decreasing.
\end{lemma}

\begin{proof}
Let us define $\eta:\mathbb{R}^d\to \R$ by
\begin{equation*}
\eta(x):=\begin{cases}
b & \text{if } \phi(x)\geq b,\\
\phi(x)-a & \text{if } a<\phi(x)<b,\\
0 & \text{if } \phi(x)\leq a.
\end{cases}
\end{equation*}
Then $\eta(x) \in W_0^{1,\infty}(\mathbb{R}^d) $, $\nabla \phi = \nabla \eta$ in $\{a<\phi(x)<b\}$, and $\nabla \eta =0$ in $\mathbb{R}^d \backslash \{a<\phi(x)<b\}$. Therefore
\begin{eqnarray*}
\text{LHS of \eqref{comp_integral}}
&=&\int_{\mathbb{R}^d} \nabla(f *(-V))\cdot \nabla \eta\\
&\leq& \int_{\mathbb{R}^d} (f^* *\Delta V) \eta^*= \int_{\{\phi^*>a\}} (f^* * \Delta V) (\max\{\phi^*, b\} -a),
\end{eqnarray*}
where the inequality comes from Riesz's rearrangement inequality.  Note that we obtain an equality if $f=f^*$ and $\eta=\eta^*$. Hence the lemma is proved.
\end{proof}

The following lemma corresponds to Theorem 17.5 in \cite{v}.

\begin{lemma} \label{rearrangement_comparison} Let $V$ be given by (B). Let $f, \bar f$ and $g$ be non-negative radially decreasing functions in $L^1(\mathbb{R}^d)$, where $ f \prec \bar f$.  Let $h>0$, and let $v_1, v_2\in D$ be two non-negative radial decreasing functions.  Assume $v_1$ and $v_2$ satisfies
\begin{equation}\label{eq1}
-h \Delta (v_1)^m - h \nabla \cdot (v_1 \nabla (f * V)) + v_1 \prec g,
\end{equation}
\begin{equation} \label{eq2}
-h \Delta (v_2)^m - h \nabla \cdot (v_2 \nabla (\bar f * V)) + v_2 = g.
\end{equation}
Then we have $v_1 \prec v_2$.
\end{lemma}
\begin{proof}
Let $u_i:= v_i^m$ and define $u:= u_1-u_2$,  $v:= v_1 - v_2, $ $A(r):=\int_{B(0,r)} v(x) dx.$ Our goal is to show $A(r)\leq 0$ for all $r\geq 0$.

Subtracting \eqref{eq1} from \eqref{eq2}, and integrating the quantity in $B(0,r)$ yields that
\begin{equation}
\int_{B(0,r)}-h\Delta udx - h\Big( v_1(r) \int_{B(0,r)} f*\Delta V dx - v_2(r)\int_{B(0,r)} \bar f*\Delta V dx\Big) + A(r) \leq 0,
\end{equation}
which can be written as
\begin{equation}
-h c_d r^{d-1} u'(r) - h v(r) \int_{B(0,r)} f*\Delta Vdx - h v_2(r) \int_{B(0,r)} \big(f- \bar f)*\Delta V dx + A(r)   \leq 0.
\end{equation}
(Here $u'(r)$ exists due to the fact that $v_i \in D$ for $i = 1,2$, which implies that $\Delta u$ is in $L^1$.)  Since we assume $f \prec \bar f$, it follows that $\int_{B(0,r)} ((f-\bar f)*\Delta V) dx \leq 0$ for all $r\geq 0$. 
Therefore
\begin{equation} \label{new_eq_u_v}
-h c_d r^{d-1} u'(r)  - h v(r) \int_{B(0,r)} f*\Delta V + A(r) \leq 0~\text{ for all } r\geq 0.
\end{equation}
 Note that since $u_i$ and $v_i$ both vanish at infinity,  from \eqref{new_eq_u_v} it follows that $\lim_{r\to\infty} A(r)\leq 0$.  Hence if $A(r)$ is positive somewhere, it achieves its positive maximum at some point $r_0>0$. At $r=r_0$ we have $v(r_0)=A'(r_0)=0$, and \eqref{new_eq_u_v} becomes
\begin{equation*}
u'(r_0) \geq \frac{A(r_0)}{hc_d r^{d-1}} >0,
\end{equation*}
which means $u_2-u_1$ is strictly increasing at $r_0$: hence $v_2-v_1$ will also be strictly positive in $(r_0, r_0+\epsilon)$ for some small $\epsilon$, which implies $A(r_0+\epsilon)>A(r_0)$. This contradicts our assumption that $A(r)$ achieves its maximum at $r_0$.  Therefore $A(r)\leq 0$ for all $r$, which means $v_2\prec v_1 $.
\end{proof}

\textbf{Proof of Proposition~\ref{rearrangement_comp_one_step}:} The proof is parallel to that of Theorem 11.7 as in \cite{v}. For any test function $\phi\in W_0^{1,\infty}(\mathbb{R}^d)$, we have
\begin{equation} \label{variational}
h\int_{\mathbb{R}^d} \nabla u^m \cdot \nabla \phi + h \int_{\mathbb{R}^d} u \nabla(f*V)\cdot \nabla \phi + \int_{\mathbb{R}^d} u \phi = \int_{\mathbb{R}^d}g\phi,
\end{equation}
where $\phi \in W_0^{1,\infty}(\mathbb{R}^d)$ is any test function. Now let us take $\phi(x) := (u^m(x)-t)_+$ where $t>0$,  and differentiate the equation with respect to $t$. Then we have:
\begin{equation}\label{I}
-\underbrace{h(\frac{d}{dt}\int_{\{u^m>t\}} |\nabla u^m|^2)}_{I_1} - \underbrace{h(\frac{d}{dt}\int_{\{u^m>t\}} \frac{m}{m+1} \nabla(f*V)\cdot \nabla (u^{m+1}))}_{I_2} + \underbrace{\int_{\{u^m>t\}}u}_{I_3} =  \underbrace{\int_{\{u^m>t\}}g}_{I_4} .
\end{equation}
Following the proof of Theorem 17.7 in \cite{v}, one can check that
\begin{eqnarray*}
I_1 &\leq& \int_{\{(u^*)^m>t\}} h\Delta ((u^*)^m)~~\text{(with equality if $u\equiv u^*$)},\\
 I_3 &=&  \int_{\{(u^*)^m>t\}}u^*,\\
 I_4 &\leq& \sup_{|\Omega|=|\{u^m>t\}|}\int_\Omega g^* = \int_{\{(u^*)^m>t\}}g^* .
 \end{eqnarray*}
It remains to examine $I_2$. Using Lemma \ref{rearrangement_comp}, it follows that
\begin{eqnarray*}
I_2 &=& h \lim_{\epsilon\to 0} \frac{1}{\epsilon} \int_{\{t<u^m<t+\epsilon\}} \frac{m}{m+1} \nabla(f*(-V))\cdot \nabla (u^{m+1})\\
&\leq& h\liminf_{\epsilon\to 0} \frac{1}{\epsilon} \int_{\{t<(u^*)^m<t+\epsilon\}}\frac{m}{m+1} (f^* * \Delta V)(\max\{u^{m+1}, (t+\epsilon)^{1+\frac{1}{m}} \} - t^{1+\frac{1}{m}})_+\\
&=& h t^{\frac{1}{m}} \int_{\{(u^*)^m>t\}} f^* * \Delta V.
\end{eqnarray*}
Plugging in the four inequalities into \eqref{I}, the following inequality holds for all $t\geq0$:
\begin{equation}
-\int_{\{(u^*)^m>t\}} h\Delta ((u^*)^m) - ht^\frac{1}{m} \int_{\{(u^*)^m>t\}} f^* * \Delta V +  \int_{\{(u^*)^m>t\}}u^* \leq  \int_{\{(u^*)^m>t\}}g^*.
\end{equation}
Since $t\geq0$ is arbitrary, the above inequality implies
\begin{equation}
-h\Delta ((u^*)^m) - h\nabla \cdot (u^* \nabla(f^* * V)) + u^* \prec g^*.
\end{equation}
On the other hand, by assumption, $\bar u$ solves
\begin{equation}
-h\Delta (\bar u^m) - h\nabla \cdot (\bar u \nabla(\bar f * V)) + \bar u = \bar g,
\end{equation}
where $\bar f \succ f^*$ and $\bar g \succ g^*$.
Note that $u \in D$ implies $u^* \in D$. So we can apply Lemma \ref{rearrangement_comparison} and get $u^*\prec \bar{u}$.
\hfill$\Box$

\end{document}